\documentclass[10pt]{amsart}
\usepackage{amssymb}
\usepackage{amsmath}
\usepackage{amsthm}
\usepackage{graphicx}
\usepackage{color}

\title[Propagation around a Lagrangian submanifold of radial points]{Propagation of singularities around a Lagrangian submanifold of radial points}
\author{Nick Haber and Andr\'{a}s Vasy}
\date{\today}
\thanks{The authors were partially supported by the Department of Defense (DoD) through the National Defense Science \& Engineering Graduate Fellowship (NDSEG) Program (NH), a National Science Foundation Graduate Research Fellowship under Grant No. DGE-0645962 (NH), the NSF under Grant No. DMS-0801226 (AV) and from a Chambers Fellowship at Stanford University (AV)}
\subjclass[2010]{35A21, 35P25}

\numberwithin{equation}{section}
\newtheorem{thm}{Theorem}[section]
\newtheorem{lem}[thm]{Lemma}

\newtheorem{cor}[thm]{Corollary}

\theoremstyle{remark}
\newtheorem*{rmk}{Remark}
\theoremstyle{definition}
\theoremstyle{definition}
\newtheorem{define}[thm]{Definition}

\newcommand{\RR}{\mathbb R}
\def\WF{\mathrm{WF}}
\def\Id{\operatorname{Id}}
\def\IL{\mathcal I_{\Lambda, U}}
\newcommand{\lrangle}[1]{\langle #1 \rangle}
\def\supp{\mathrm{supp}}
\def\sgn{\mathrm{sgn}}
\def\Ell{\mathrm{Ell}}
\newcommand{\hS}{\hat \Sigma}
\newcommand{\MM}{\mathcal M}
\newcommand{\M}{\MM_\Lambda}

\begin{document}

\begin{abstract}
In this work we study the wavefront set of a solution $u$ to $Pu = f$, where $P$ is a pseudodifferential operator on a manifold with real-valued
homogeneous principal symbol $p$, when the Hamilton vector field
corresponding to $p$ is radial on a Lagrangian submanifold $\Lambda$ contained in the characteristic set of $P$. The standard propagation of singularities theorem of Duistermaat-H\"ormander gives no information at $\Lambda$. By adapting the standard positive-commutator estimate proof
of this theorem, we are able to conclude additional regularity at a point $q$ in this radial set, assuming some regularity around this point. That is, the a priori assumption is either a weaker regularity assumption at $q$, or a regularity assumption near but not at $q$. Earlier results
of Melrose and Vasy give a more global version of such analysis. Given some regularity assumptions around the Lagrangian submanifold, they obtain some regularity at the Lagrangian submanifold. This paper
microlocalizes these results, assuming and concluding regularity only at a
particular point of interest. We then proceed to prove an analogous result,
useful in scattering theory, followed by analogous results in the context of Lagrangian regularity.
\end{abstract}

\nocite{*}

\maketitle

\section{Introduction} \label{sect:intro}
This paper studies the wavefront set of a solution $u$ to $Pu = f$, where $P$ is a pseudodifferential operator on a manifold with real-valued homogeneous principal symbol $p$, when the Hamilton vector field corresponding to $p$ is radial on a Lagrangian submanifold contained in the characteristic set of $P$. According to a theorem of Duistermaat-H\"{o}rmander (\cite{duistermaathormander}), singularities propagate along bicharacteristics of this Hamilton vector field. This theorem gives us no information about the wavefront set when the Hamilton vector field is radial. \cite{melroseasymp} and \cite{kerr-desitter} give a global analysis of the propagation of singularities around a Lagrangian submanifold of radial points. By adapting the standard positive commutator estimate proof of this theorem, we microlocalize these results. 

After proving such a result, we proceed to prove an analogous result, useful in scattering theory, in particular in resolvent estimates. Analogous to the standard propagation of singularities, microlocal Sobolev bounds on $u_\tau$ which are uniform in $\tau \in [0, 1]$ or $(0, 1]$ propagate forward along bicharacteristics, assuming uniform Sobolev bounds for $(P - i \tau) u$, where now $P$ is of order $0$ (see, for instance, \cite{melroseasymp}). We prove a corresponding statement around a Lagrangian submanifold of radial points, generalizing to solutions of $P - i Q_\tau$, with $P, Q_\tau$ of equal order (not necessarily $0$), with suitable boundedness and positivity assumptions on $Q_\tau$. This is again a microlocal result which generalizes a global result given in \cite{melroseasymp}.

Lastly, we prove analogs in the context of Lagrangian regularity, essentially replacing ``$u$ is microlocally $H^s(X)$'' with ``$u$ is microlocally a Lagrangian distribution.'' This follows the analyses of \cite{hmv1} and \cite{hmv2}.

It should be emphasized that these results are completely local. That is, in order to conclude regularity for $u$ at a point $q$ in this Lagrangian submanifold, we need only have regularity for $f$ in an arbitrarily small neighborhood of $q$. At times we also need regularity assumptions on $u$ around the bicharacteristics approaching the smallest conic subset containing the $\RR_+$-orbit containing $q$, and at other times we also need a priori regularity assumptions on $u$ - it is important to note that these requirements are again local around $q$. Thus we do not, for instance, require regularity assumptions around the whole Lagrangian submanifold.

Under the nondegeneracy assumption $dp \neq 0$, the largest-dimensional subspace on which a Hamilton vector field can be radial is a Lagrangian submanifold. This occurs naturally in many applications, including geometric scattering theory. Indeed, these results generalize a series of results in \cite{melroseasymp}. For the treatment of the opposite extreme, that is, that of an isolated radial point, see for instance  \cite{guilleminschaeffer}, \cite{hmv1}, and  \cite{hmv2}.

In Section~\ref{sect:setup}, we introduce basic microlocal terminology. We then state the standard (principal-type) propagation of singularities theorems and discuss radial points in Section~\ref{sect:principalpropagation}. In Section~\ref{sect:cosphere}, we discuss the cosphere bundle as a quotient of the cotangent bundle (excluding the zero section). As it is at times easier to discuss dynamics on the cosphere bundle than it is on the cotangent bundle, we regard certain conic sets, such as wavefront sets, to be subsets of the cosphere bundle. We then state the main theorems of the paper in Section~\ref{sect:theoremstatements}. In Section~\ref{sect:sketch}, we sketch the proofs of these theorems. The theorems contain 'threshold' values ($s_0, s_1$) that have explicit values which are complicated to state in generality but can be refined considerably under additional assumptions. We thus delay discussing these values until Section~\ref{sect:s0s1formulas}. We then proceed to prove Theorem~\ref{thm:general} in Sections~\ref{sect:classicaldynamics}, \ref{sect:commutator}, and \ref{sect:operatorconstruction}. Theorem~\ref{thm:classical} follows as a special case. In section~\ref{sect:classicaldynamics}, we analyze the Hamiltonian dynamics around the radial points. In Section~\ref{sect:commutator}, we give the positive commutator proof of Theorem~\ref{thm:general}, assuming the existence of certain operators. In Section~\ref{sect:operatorconstruction}, we construct these operators. In Section~\ref{sect:tauproof}, we adapt these constructions for Theorem~\ref{thm:tau}. In Section~\ref{sect:iterativeregularity}, we review the notion from \cite{hmv1} of iterative regularity, in the context of Lagrangian regularity, state and prove Theorems~\ref{thm:iterative} and \ref{thm:tauiterative}.

In proving these theorems we make arguments which are intended to be adaptable to other situations. In particular, it may be possible to find a more explicit normal form for the Hamilton vector field around a Lagrangian submanifold of radial points, with Lemmas~\ref{lem:coordinates} and \ref{lem:geometriclemma} as easy consequences. These lemmas are, however, closer to the bare minimum needed to prove the main theorems, and thus indicate how the proofs might be adapted in cases where such a normal form cannot be found. As remarked after Lemmas~\ref{lem:s0operators} and \ref{lem:s1operators}, we can assure that certain error terms ($F_t$) are smoothing, which is stronger than the lemma statements. This is, however, not needed for the proof of Theorem~\ref{thm:general}, and requires a bit more work. An analogous error term improvement is needed in the proof of Theorem~\ref{thm:tau}, and we prove this in Section~\ref{sect:tausource}.

\subsection{Basic Setup} \label{sect:setup}

We recall several definitions so as to fix notation. Analysis will take place on $X$, an $n$-dimensional manifold without boundary. Given $P \in \Psi^m(X)$, the $m$th order pseudodifferential operators on $X$, we let $$\sigma_m(P) \in S^m(T^*X)/S^{m-1}(T^*X)$$ denote the principal symbol of $P$, where $S^m(T^*X)$ is the set of $m$-th order Kohn-Nirenberg symbols on $T^*X$.

Let $o$ be the 0-section of $T^*X$. Denote by $\mu: T^*X\backslash o \times \RR_+ \rightarrow T^*X \backslash o$ the natural dilation of the fibers of $T^*X\backslash o$: given $v \in T^*_xX, v \neq 0$, $\mu((x, v), t) = (x, tv)$. We call a subset of $T^*X \backslash o$ conic if $\mu$ acts on it. We call a function $f$ on $T^*X \backslash o$ homogeneous of order $m$ if $$(\mu( \cdot, t)^* f)(x, v) = t^m f(x, v)$$ and a vector field $V$ on $T^*X \backslash o$ homogeneous of order $m$ if $$(\mu^{-1}(\cdot, t)_*) V(x, v) = t^mV(x, v).$$ At times we will assume that $P \in \Psi^m(X)$ has a homogeneous (of order $m$) principal symbol $p$ (i.e. a homogeneous representative for $\sigma_m(P)$ - note that, if this exists, it is unique), defined on $T^*X \backslash o$. Given such a $p$, real-valued, we let $H_p$ be the associated Hamilton vector field on $T^*X \backslash o$. Note that then $H_p$ is homogeneous of order $m - 1$.

Given $P \in \Psi^m(X)$, let $\Sigma(P) \subset T^*X \backslash o$ denote the characteristic set of $P$, and let $\mathrm{Ell}(P) \subset T^*X \backslash o$ be the complement. Note that if we assume that $P$ has a homogeneous principal symbol $p$, $\Sigma(P) = p^{-1}(0)$. Given $u \in \mathcal D'(X)$, we let $$\WF^s(u) = \bigcap_{A \in \Psi^s(u), Au \in L^2_{\mathrm{loc}}(X)} \Sigma(A)$$ be the Sobolev wavefront sets of $u$. That is, $q \notin \WF^s(u)$ if there exists an $A \in \Psi^s(X)$, elliptic at $q$, with $Au \in L^2_{loc}(X)$. 

Given $A \in \Psi^m(X)$, let $\WF'(A)$ be the microsupport of $A$, that is $q \notin \WF'(A)$ if there exists a $B \in \Psi^0(X)$ with $q \in \Ell(B)$ and $BA \in \Psi^{-\infty}$. If $A_t \in L^\infty([0, 1]_t, \Psi^m)$, then $$q \notin \WF'(A)$$ if there exists such a $B$ with $BA_t \in L^\infty([0, 1], \Psi^{-\infty}(X))$. Similarly, given $a \in S^m(T^*X)$, we let the essential support of $a$ be denoted by $\mathrm{esssup}(a) \subseteq T^*X \backslash o$, that is, 
$$q \notin \mathrm{esssup}(a)$$
if there is a conic open neighborhood of $q$ on which $a$ satisfies order $-\infty$ bounds. Given $a_t \in L^\infty([0, 1]_t, S^m(T^*X))$, let 
$$q \notin \mathrm{esssup}_{L^\infty([0,1])}(a_t)$$
 if there is a conic open neighborhood of $q$ on which $a_t$ satisfies order $-\infty$ bounds independent of $t$. 

If $u = (u_\tau)_{\tau \in [0, 1]} \in L^\infty([0, 1], \mathcal D'(X))$, then say that $$q \notin \WF^s_{L^\infty([0, 1])}(u)$$ if there exists an $A \in \Psi^s(X)$ with $q \in \Ell(A)$ and $Au_\tau \in L^\infty([0, 1]_\tau, L^2_{\mathrm{loc}}(X))$; with the obvious modification if $u = (u_\tau)_{\tau \in (0, 1]}$. We can relax the requirement of a fixed $A$, making it $\tau$-dependent, as follows. Given $A_\tau \in L^\infty([0, 1]_\tau, \Psi^s(X))$ with choice of principal symbol $a_\tau \in L^\infty([0, 1], S^s(T^*X))$, then in local coordinates $(x, \xi)$, we say that $$(\hat{x}, \hat{\xi}) \in \Ell_{L^\infty([0, 1])}(A_\tau)$$ if, in a conic neighborhood $U \subset T^*X \backslash o$ of $\hat x, \hat \xi$, $|a_\tau(x, \xi)| \geq C\lrangle{\xi}^s$ for sufficiently large $\xi$, with $C$ and $U$ independent of $\tau$. We then have $q \notin \WF^s_{L^\infty([0, 1])}(u)$ if there is such an $A_\tau$ with $q \in \Ell_{L^\infty([0, 1])}(A_\tau)$.

Note that all the sets defined in the preceding paragraphs are conic subsets of $T^*X \backslash o$. Shortly, we shall regard them as subsets of the cosphere bundle - more on those in Section~\ref{sect:cosphere}.

\subsection{Standard Propagation of Singularities} \label{sect:principalpropagation}

We now recall a standard result (\cite{duistermaathormander}). As is customary, we refer to the integral curves of $H_p$ as bicharacteristics. We do not limit ourselves to bicharacteristics within $\Sigma(P)$ when using this term; we will specify inclusion in $\Sigma(P)$ in theorem statements.

\begin{thm}(Duistermaat-H\"{o}rmander, \cite[Theorem 6.1.1']{duistermaathormander}) \label{thm:hormander}
Suppose $P \in \Psi^m(X)$ with real-valued homogeneous principal symbol $p$. Then given $u \in \mathcal D'(X)$, $$\WF^s(u) \backslash \WF^{s - m + 1}(Pu)$$ is a union of maximally extended bicharacteristics in $\Sigma(P) \backslash \WF^{s - m + 1}(Pu)$. \qed
\end{thm}

We now recall an analogous result, useful in scattering theory. Statements similar to this can be found in many places (see, for instance, \cite{melroseasymp}). The semiclassical version of this is proved in \cite{datchevvasygluing} (Lemma 5.1), and the proof carries over without difficulty.

\begin{thm}(Datchev-Vasy, \cite[Lemma 5.1]{datchevvasygluing})\label{thm:melrosetau}
Given $P \in \Psi^m(X), Q = (Q_\tau)_{\tau \in [0, 1]} \in L^\infty([0, 1]_\tau, \Psi^m(X))$, and $u_\tau \in L^\infty([0, 1]_\tau, \mathcal D'(X))$ such that
\begin{itemize}
	\item $P$ has real-valued homogeneous principal symbol $p$,
	\item $Q_\tau$ has real-valued (choice of) principal symbol $q_\tau \geq 0$
	\item $P - i Q_\tau$ is elliptic for $\tau > 0$ (so in particular we can choose $q_\tau > 0$ for $\tau > 0$). 
\end{itemize}
then $$\WF^s_{L^\infty([0, 1])}(u_\tau) \backslash \WF^{s - m + 1}_{L^\infty([0, 1])}((P - i Q_\tau) u_\tau)$$ is a union of maximally backward-extended bicharacteristics in $$\Sigma(P) \backslash \WF^{s - m + 1}_{L^\infty([0, 1])}((P - i Q_\tau) u_\tau).$$ \qed
\end{thm}
Note that, while regularity propagates both forward and backward along bicharacteristics in Theorem~\ref{thm:hormander}, regularity only propagates forward along bicharacteristics in Theorem~\ref{thm:melrosetau}.

\begin{define}
We call the vector field $f(\cdot) \mapsto \frac{d}{dt}|_{t = 0} f(\mu(\cdot, t))$ the radial vector field. We say that $H_p$ is {\em radial} at a point $q \in T^*X$ if $H_p$ is a scalar multiple of the radial vector field at $q$, and we then call $q$ a {\em radial point} of $H_p$.
\end{define}

Equivalently, if we choose local canonical coordinates $(x, \xi)$ for $T^*X$, then $H_p$ is radial at $q$ if it is a scalar multiple of $\xi \cdot \partial_\xi$ at $q$. Note then that Theorem~\ref{thm:hormander} and Theorem~\ref{thm:melrosetau} say nothing at radial points: if $q$ is a radial point, then $H_p$ is also radial along $q$'s orbit under $\mu$ (by the homogeneity of $H_p$). Thus the bicharacteristic going through $q$ is a conic set. As $\WF^s(u)\backslash \WF^{s - m + 1}$ is conic, the theorem says nothing here.

It turns out, however, that we can conclude more about the regularity of $u$ at the $\RR_+$-orbit of a radial point $q$, depending on the dynamics of $H_p$ around the orbit of $q$. We restrict our attention to the case where $H_p$ is radial on a Lagrangian submanifold $\Lambda$ of $T^*X \backslash o$.  Before stating the results, we make some further definitions (some slightly nonstandard) to avoid making statements in terms of the $\RR_+$ orbit of a point or bicharacteristics approaching such an orbit.

\subsection{The cosphere bundle picture} \label{sect:cosphere}

Let $$\kappa: T^*X \backslash o \rightarrow (T^*X \backslash o)/\RR_+ = S^*X$$ be the quotient map identifying the orbits of $\mu$. We identify $(T^*X \backslash o)/\RR_+$ with $S^*X$, the cosphere bundle of $X$. Given $q \in S^*X$, let $U$ be a conic neighborhood of $\kappa^{-1}(q)$ with  $\zeta: U \rightarrow \RR_+$, homogeneous of degree $1$ and nonvanishing. On $U$, we can then define the vector field $W_p = \zeta^{1 - m}H_p$. This is then homogeneous of degree $0$, so it pushes forward to a vector field on $\kappa(U) \subset S^*X$ (which we will at times also call $W_p$). Note then that $H_p$ is radial at $\kappa^{-1}(q)$ if and only if $\kappa_* W_p$ vanishes at $q$.

It is of course possible to have such a $\zeta$ globally defined on $T^*X \backslash o$ (we can, for instance, let $\zeta$ be the norm on the cotangent fibers induced by the choice of a Riemannian metric), and thus taking a globally well-defined $W_p$ and $q \in S^*X$, let $$\Gamma_q = \{x \in S^*X \backslash q \: | \lim_{t \rightarrow \infty} \exp(\pm t \kappa_* W_p)(x) = q\}.$$ Note that while $W_p$ depends on the choice of $\zeta$, the integral curves do not (different choices of $\zeta$ correspond to different parameterizations of these integral curves). In particular, if we define $\zeta$ only locally, then the integral curves of the locally-defined $W_p$ agree with the globally-defined ones.

If $U$ has a coordinate chart $\phi = \phi_0 \times \zeta: U \rightarrow V_0 \times \RR_+$ where $\phi_0$ is homogeneous of order $0$ (and $\zeta$ homogeneous of order $1$), then $\partial_\zeta$ is radial. If we set $U_0 = \kappa(U)$, then $\phi_0$ induces a coordinate chart $\psi_0: U_0 \rightarrow V_0$ determined by $\psi_0 \circ \kappa = \phi_0$.

Since $\WF^s(u), \WF'(A), \mathrm{Ell}(A)$, $\mathrm{esssup}(a)$  and their $L^\infty([0, 1])$-counterparts are conic subsets, it is natural to regard these as subsets of the cosphere bundle, and from here on we elect to do so: $$\WF^s(u), \WF'(A), \mathrm{Ell}(A), \mathrm{esssup}(a) \subset S^*X.$$  We set, for $P \in \Psi^m(X)$, $\hat \Sigma(P) = \kappa(\Sigma(P))$, and fixing a Lagrangian submanifold $\Lambda$ of $T^*X$, $L = \kappa(\Lambda)$.

Assuming $H_p$ is radial on a Lagrangian submanifold $\Lambda \subset T^*X$, $W_p$ vanishes on $L$. If we assume further that $dp \neq 0$ on $\Lambda$, then $L$ is either a sink or a source for $W_p|_{\hS}$. In fact, if we look at the linearization of $W_p|_{\hS}$ at a point $q \in L$, it has two eigenvalues: a nonzero $\lambda_0$ corresponding to the conormal bundle of $L$, and $0$. We will see this quite explicitly in Section~\ref{sect:coordinates}; for a more general discussion on why this must be true, see, for instance, \cite[Section 2]{hmv2}.

\subsubsection{The compactified cotangent bundle picture} \label{sect:compactified}

This section is optional and is included in order to give a nice picture of the classical ($H_p$) dynamics involved. In further sections we will work in the cotangent bundle and the cosphere bundle, and use this for supplementary commentary. We denote by $\overline{T}^*X$ the (fibre-) compactified cotangent bundle of $X$. See \cite{melroseasymp} for an introduction to this, and in particular a proof that it is globally well-defined; here we simply state the essential properties of it and give a local coordinate chart.

$\overline{T}^*X$ is a disk bundle over $X$, constructed by compactifying each fiber of $T^*X$ to a (closed) disk. There is an inclusion $j: T^*X \hookrightarrow \overline{T}^*X$, and the boundary $\partial \overline{T}^*X$ can be identified with the cosphere bundle $S^*X$. Given $q \in S^*X$, along with conic open neighborhood $U = \kappa^{-1}(U_0) \subset T^*X \backslash o$ of $\kappa^{-1}(q)$ and coordinate chart $\phi$ as above, we can give a coordinate chart $\varphi: \tilde{U} \rightarrow U_0 \times [0, 1]_x$ for an open neighborhood $\tilde{U} \subset \overline{T}^*X$ of $q$ as follows. Given $w \in U$, let $\varphi(j(w)) = (\phi_0(w), \frac{1}{\zeta}(w))$, and for $w \in S^*X$, let $\varphi(w) = (\phi_0(\kappa^{-1}(w)), 0)$. We have a boundary-defining function $x$ defined by $x = \frac{1}{\zeta}$ on the interior and $x = 0$ on the boundary.

Again taking the vector field $W_p = \zeta^{1 - m} H_p$ defined on $U$, $W_p$ extends uniquely (see \cite{melroseasymp}) to a vector field on $\tilde{U}$, which we will also denote by $W_p$. $W_p$ is tangent to the boundary $\partial \overline{T}^*X$, i.e. $W_p \in \mathcal V_b(\overline U)$ (the Lie algebra of vector fields tangent to the boundary). $W_p|_{S^*X}$ then agrees with $\kappa_* W_p$ as defined in Section~\ref{sect:cosphere}.

As noted at the end of Section~\ref{sect:cosphere}, $L$ is either a sink or a source for $W_p|_{\hS}$. As we will see in Section~\ref{sect:coordinates}, more is true: $L$ is in fact a sink or a source for $W_p|_{\overline \Sigma}$, where $\overline \Sigma = \Sigma \cup \hS \subset \overline T^*X$. The linearization of this has the same eigenvalue $\lambda_0$ corresponding to any boundary defining function. Our proofs of the following theorems depend on the behavior of $W_p$ near $L$ not just in the cosphere bundle but also in the interior of $\overline{T}^*X$, and that $L$ is a source or sink in this sense will be very important.

\subsection{Statement of Theorems} \label{sect:theoremstatements}

First, we state a simple version, valid, for instance, when $P \in \Psi^m_{\mathrm{cl}}(X)$. Here we choose a density on $X$ in order to define $P^*$; however $s_0$ in the statement below does not depend on this choice. See section~\ref{sect:s0s1formulas} for more on this independence. In particular, the homogeneous requirement on $\sigma_{m-1}(\frac{P - P^*}{2i})$ does not depend on the choice of density.

\begin{thm} \label{thm:classical}
Given $P \in \Psi^m(X)$ with a real-valued homogeneous principal symbol $p$ such that $H_p$ is radial (and nonvanishing) on a conic Lagrangian submanifold $\Lambda \subset \Sigma(P)$, with the additional assumption that $\sigma_{m-1}(\frac{P - P^*}{2i})$ has homogeneous representative, then given $q \in \kappa(\Lambda)$, there exist $s_0 \in \RR$ such that
\begin{itemize}
	\item For $s < s_0$,  if there is an open neighborhood $U_0 \subset S^*X$ of $q$ disjoint from $\WF^{s - m + 1}(Pu)$ and from $\Gamma_q \cap \WF^s(u)$,
then $q \notin \WF^s(u)$.
	\item For every  $s > s_1 > s_0$, $q \notin \WF^{s_1}(u)$ implies $q \notin \WF^s(u)\backslash \WF^{s - m + 1}(Pu)$.
\end{itemize}
\end{thm}

Next, we state a more general version, taking away the assumption on $\sigma_{m-1}(\frac{P - P^*}{2i})$.


\begin{thm} \label{thm:general}
Given $P \in \Psi^m(X)$ with a real-valued homogeneous principal symbol $p$ such that $H_p$ is radial (and nonvanishing) on a conic Lagrangian submanifold $\Lambda \subset \Sigma(P)$, then given $q \in \kappa(\Lambda)$, there exist $s_0, s_1 \in \RR$ such that
\begin{itemize}
	\item For $s < s_0$, if there is an open neighborhood $U_0 \subset S^*X$ of $q$ disjoint from $\WF^{s - m + 1}(Pu)$ and from $\Gamma_q \cap \WF^s(u)$, then $q \notin \WF^s(u)$.
	\item If $s > s_1$,then $q \notin \WF^{s_1}(u)$ implies $q \notin \WF^s(u)\backslash \WF^{s - m + 1}(Pu)$.
\end{itemize}
\end{thm}

We end this section by stating a theorem useful in scattering theory. As noted above, $L$ is either a submanifold of sinks or a submanifold of sources for $W_p$. As a technical assumption, we take as given a choice of density on $X$. This is needed for the positive-semidefinite assumption below; as discussed below, some more effort should allow this to be removed.

\begin{thm} \label{thm:tau}
Given $P \in \Psi^m(X), Q = (Q_\tau)_{\tau \in [0, 1]} \in L^\infty([0, 1]_\tau, \Psi^m(X))$, and \\ $u_\tau \in L^\infty([0, 1]_\tau, \mathcal D'(X))$
 such that
\begin{itemize}
	\item $P$ has a real-valued homogeneous principal symbol $p$ such that $H_p$ is radial (and nonvanishing) on a conic Lagrangian submanifold $\Lambda \subset \Sigma(P)$,
	\item $Q_\tau$ is positive-semidefinite for $\tau > 0$, and
	\item $P -  i Q_\tau$ is elliptic for $\tau > 0$,
\end{itemize}
then for $q \in \kappa(L)$, there exist $s_0, s_1 \in \RR$ such that
\begin{itemize}
	\item if $\kappa(\Lambda)$ is a sink for $W_p|_{S^*X}$, then for $s < s_0$, the existence of an open neighborhood $U_0 \subset S^*X$ of $q$ disjoint from $$\WF_{L^\infty([0, 1])}^{s - m + 1}((P - i Q_\tau)u_\tau)$$ and from $$\Gamma_q \cap \WF_{L^\infty([0, 1])}^s(u_\tau)$$  implies $q \notin \WF_{L^\infty([0, 1])}^s(u_\tau)$.
	\item if $\kappa(\Lambda)$ is a source for $W_p|_{S^*X}$, then for $s > s_1$,  $$q \notin \WF_{L^\infty([0, 1])}(u_\tau) \backslash \WF_{L^\infty([0,1])}^{s - m + 1}((P - i Q_\tau)u_\tau).$$
\end{itemize}
\end{thm}

The value of $s_0$ for Theorem~\ref{thm:tau} is the same as in Theorem~\ref{thm:general}, and $s_1$ can be taken to be the lower bound on what $s_1$ can be in Theorem~\ref{thm:general}. 

It is worth noting that we can relax the assumption that $Q_\tau$ is positive-semidefinite for $\tau > 0$. If we have a choice of $\sigma_m(Q_\tau)$ that is positive for $\tau > 0$, then we would like to be able to apply a sharp G\r{a}rding inequality $Q_\tau \geq Q'_\tau$ for $Q'_\tau$ of lower order. If we can make $Q_\tau'$ independent of $\tau$, or at least give it some uniform control in $\tau$, then $Q'$ can then be absorbed in $P$, and the net effect would be a shift in $s_0$ and $s_1$. We elect not to pursue such a uniform sharp G\r{a}rding inequality in this paper, as it is besides the central point. It is easier to relax this positive-semidefinite assumption in special circumstances, though. If, for instance, $Q_\tau = \tau Q$ with a choice of $\sigma_m(Q)$ positive, then we may simply apply sharp G\r{a}rding or a related construction and again absorb a term into $P$.

For all three theorems, $s_0$ and $s_1$ are determined entirely by $\sigma_{m-1}(\frac{P - P^*}{2i})$ and $dp$ around $\kappa^{-1}(q)$. We give explicit formulas for them in Section~\ref{sect:s0s1formulas}, but it is helpful to motivate their formulas in the following sketch.

As mentioned above, two more theorems are contained in Section~\ref{sect:iterativeregularity}. As the results require further definitions, we postpone their statements until then.

\subsection{Sketch of Proofs} \label{sect:sketch}

In this section, we sketch the proofs that follow. This should help motivate the theorem statements, as well as help the reader to separate the essential details of the proofs from the technical details which can be arranged more easily. In the proofs of these statements, we adapt the positive commutator argument that is now standard in microlocal analysis (see, for instance, \cite{hormanderpositivecommutator}, Proposition 3.5.1).  

In particular, in order to prove Theorem~\ref{thm:general}, we would like to construct families of pseudodifferential operators $A_t, G_{1, t}, G_{2, t}, E_t, F_t$ so that $$\frac{1}{2i}(A_t P - P^* A_t) = \pm(G_{1, t}^2 + G_{2, t}^2) + E_t + F_t,$$ where all are of acceptably low order when $t > 0$ (we can take order $-\infty$ when $s < s_0$, but when $s > s_1$, we must stay at or above this threshold). That way, we can make sense of the following pairing with $u$ for $t > 0$:
\begin{align*}
\frac{1}{2i}  \lrangle{u, (A_t P - P^* A_t) u} & = \lrangle{u, (\pm(G_{1, t}^* G_{1, t}+ G_{2, t}^* G_{2, t}) + E_t + F_t) u} \\
\mathrm{Im}(\lrangle{A_t u, P u}) & = \pm(\|G_{1, t} u\| ^2+ \|G_{2, t} u\|^2) + \lrangle{u, E_t u} + \lrangle{u, F_t u}
\end{align*}
We have implicitly chosen (in writing these inner products and $P^*$) ana density for $X$ - as we will argue later, it does not matter which. As $t \rightarrow 0$, we would like $G_{2, t}$ to approach an operator of order $s$, elliptic at the point $q \in L$ which we would like to prove is not in $\WF^s(u)$. This is accomplished if we can bound the left hand side of the above equation, as well as $\lrangle{u, E_t u}$ and $\lrangle{u, F_t u}$. We bound the left hand side by requiring that $A_t$ not only have the correct order but also microsupport contained in some open neighborhood $U_0 \subset S^*X$ on which we assume that $Pu$ has regularity. We bound the $E_t$ term by requiring that $E_t$ have microsupport contained in some neighborhood where we can assume $u$ has regularity. Lastly, we bound the $F_t$ term by requiring that $F_t$ has order $2s -1$, and work by induction, assuming that $U_0 \cap \WF^{s -\frac{1}{2}}(u) = \emptyset$. Notice that $\|G_{1, t} u\|^2$ gets bounded for free here, since it is of the same sign as $\|G_{2, t} u\|^2$. Its inclusion is simply meant to make the operator constructions easier. 

We construct these operators by quantizing real-valued symbols $a_t, g_{1, t}, g_{2, t}, e_t$ so that 
$$\frac{1}{2} H_p a_t + \sigma_{m-1}\bigl(\frac{P - P^*}{2i}\bigr) a_t = \pm(g_{1, t}^2 + g_{2, t}^2)+ e_t.$$
 To do that, we further assume that $U$ has a coordinate chart $\phi = \phi_0 \times \zeta: U \rightarrow V_0 \times \RR_+$ as in Section~\ref{sect:cosphere}. In order to localize to $U$, we take 
$$a_t = (\chi (\phi_0))^2  (\rho_t(\zeta))^2.$$ Here $\chi: V_0 \rightarrow \RR$ is a  cutoff function localizing to $U$, and $\rho_t$  gives us the correct order properties (so for $t = 0$, it is the correct power of $\zeta$, and for $t > 0$, of suitably lower order in $\zeta$). Taking $a_t$ to be a square fixes its sign; as argued in the sketch of Theorem~\ref{thm:tau}, there is a better reason for making this a square.

Define $$\lambda = - H_p \zeta.$$ This is a symbol, homogeneous of degree $m-1$, defined on $U$. Under the assumption that $dp \neq 0$ on $\Lambda$ (and hence that $H_p \neq 0$ on $\Lambda$), we may assume, possibly after shrinking $U$, that $\lambda$ is elliptic on $U$. The thresholds $s_0, s_1$ are chosen precisely so that when $s < s_0$,
\begin{align} \label{eq:thresholdformula}
\chi(\phi_0)^2 \biggl(\frac{1}{2} H_p (\rho_t(\zeta)^2) + \sigma_{m-1}\bigl(\frac{P - P^*}{2i}\bigr) \rho_t(\zeta)^2 \biggr)
\end{align}
and $\lambda$ are of the same sign, and when $s > s_1$, they are of opposite sign. For both cases, we need only have this condition satisfied for $\zeta$ sufficiently large (as we only need to determine our operators up to order $-\infty$), and we may also shrink $U$. We develop explicit formulae in Section~\ref{sect:s0s1formulas}.

Note that since $H_p$ is radial at $q$, $H_p (\chi(\phi_0)^2)$ must vanish at $q$, so the $$\frac{1}{2} \chi (\phi_0)^2 H_p (\rho_t (\zeta)^2)$$ term must be what contributes to $g_{2, t}^2$. This also explains the inclusion of $$\sigma_{m-1}\bigl(\frac{P - P^*}{2i}\bigr) a_t$$ in the definitions of $s_0$ and $s_1$ above. As we assume no control over this term, we must dominate it by $\frac{1}{2} H_p a_t$. In the positive commutator proof of Theorem~\ref{thm:hormander}, we can dominate this term with $\rho_t^2 H_p (\chi (\phi_0)^2)$, but here, this is not an option, and we must rely on the growth rate of $\rho_t$ to dominate this term.

As we shall see in Section~\ref{sect:coordinates}, $\kappa_* W_p|_{hS} (= \kappa_* (\zeta^{1 - m} H_p)|_{\hS}$) is a sink or source depending on whether $\lambda$ is negative or positive (see also Section~\ref{sect:compactifiedcontinued}, for what is perhaps a clearer picture). Thus in the $s < s_0$ case, the sign of \eqref{eq:thresholdformula} does not match with the sign of $\rho_t H_p (\chi \circ \phi_0)$ everywhere. We must then have the regularity assumption on $u$ in some deleted neighborhood of $q$. As we will show below in Section~\ref{sect:geometriclemma}, this amounts to assuming regularity on bicharacteristics which approach $q$, as stated in Theorem~\ref{thm:general}. When $s>s_1$, the signs of these two terms can be made to match everywhere on the characteristic set, and we no longer need this assumption. In regularizing, however, we cannot pass through this threshold $s_1$ as $t \rightarrow 0$, as the sign would switch, taking away any hope of getting the desired bound.  Thus we need an a priori regularity assumption $q \notin \WF^{s_1}(u)$, and we regularize from that level. This a priori assumption also allows us to have the inductive assumption $U \cap \WF^{s - \frac{1}{2}}(u) = \emptyset$, as we can start the induction at $s = s_1 + \frac{1}{2}$ (if the expected conclusion is to be stronger than this), but as we shall see below in Section~\ref{sect:tausource}, this is for convenience rather than necessity, as we can actually take $F_t \in \Psi^{-\infty}(X)$.

We use a similar argument in the proof of Theorem~\ref{thm:tau}. One key difference is that, by assumption, $P - i Q_\tau$ is elliptic for $\tau > 0$, so elliptic regularity implies some regularity for $u_\tau$ for $\tau > 0$. In a sense, this regularizes for us, and we can use our limiting, $t = 0$, operators (hence in what follows we take away the subscripts and write, for instance, $A$ for $A_0$). This allows us to take away the a priori assumption $q \notin \WF^{s_1}(u)$. In taking away this a priori regularity, we can no longer have the inductive assumption  $U \cap \WF^{s - \frac{1}{2}}(u) = \emptyset$ (which would control the $\lrangle{u, F_0 u}$ term), as we cannot start our induction at $s = s_1 - \frac{1}{2}$. As mentioned above, though, with greater care in symbol construction, we can actually take $F \in \Psi^{-\infty}$, so this is not a real issue.

Another key difference is that, as we have regularity on $(P - i Q_\tau)u$, we modify the argument to involve the ``commutator'' $\frac{1}{2i}(A (P - i Q_\tau) - (P^* + i Q_\tau)A)$ (note that since $Q_\tau$ is positive semidefinite, we assume $Q_\tau^* = Q_\tau$). This gives us an extra term:
\begin{align*}
\frac{1}{2i}(A (P - i Q_\tau) - (P^* + i Q_\tau)A)  & = \frac{1}{2i}(A P - P^* A) - \frac{1}{2}(A Q_\tau + Q_\tau A) \\
\end{align*}
We must be careful with this extra term: for $\tau > 0$, it is one order higher than we would like $G_2$ to be, so the only way to control it is to ensure that it is, up to two orders lower, of the same sign as $\pm(G_1^2  + G_2^2)$ (i.e., have them both positive semidefinite or negative semidefinite). We chose, arbitrarily, $A$ to have nonnegative principal symbol, but in order to get another order of control, we take (as this is easy to arrange) $A = B^2$ with $B^* = B$. We can then construct operators so that: $$\frac{1}{2i}(B^2 (P - i Q_\tau) - (P^* + i Q_\tau) B^2) = \pm(G_1^2 + G_2^2) - B Q_\tau B + E + F_\tau.$$ $B Q_\tau B$ is a positive semidefinite operator, so we must ensure that the $\pm$ above is a $-$.  As a result, the $s < s_0$ argument works only when $\lambda < 0$, and the $s > s_1$ argument works only when $\lambda > 0$; hence the sink/source assumptions in the statement of Theorem~\ref{thm:tau}.

In order so that we do not need to construct $A_t$ for the proof of Theorem~\ref{thm:general} and then go back and construct $B$ so that $A_0 = B^2$ for the proof of Theorem~\ref{thm:tau}, we simply work with $B_t$, the quantization of $b_t$, throughout.

\subsubsection{Explicit formulas for $s_0, s_1$} \label{sect:s0s1formulas}

Here we give explicit formulas for the thresholds $s_0$ and $s_1$, using the coordinates and definitions of Section~\ref{sect:sketch}. At the end of the section, we argue that the formulas are independent of choices made. In the formulas below, we choose any representative for $\sigma_{m-1}(\frac{P - P^*}{2i})$ and write it simply as $\sigma_{m-1}(\frac{P - P^*}{2i})$. In the homogeneous case of Theorem~\ref{thm:classical}, the homogeneous choice is unique.

We start by determining the values for Theorem~\ref{thm:general}. As noted in the above sketch, we choose $s_0$ so that \eqref{eq:thresholdformula} remains the same sign as $\lambda$ on $U$, for all $t \in [0, 1]$, and we choose $s_1$ so that it has sign opposite to that of $\lambda$. This does not depend on the form of $\rho_t$, but only on its order of growth in $\zeta$. A quick calculation verifies that at a point $w \in U$, the critical order is the following:
$$f(w) := \frac{\sigma_{m-1}(\frac{P - P^*}{2i}) \zeta}{\lambda}(w)$$
That is, at a point $w \in U$, \eqref{eq:thresholdformula} is the same sign as $\lambda$ if and only if $\frac{\rho_t'}{\rho_t}(\zeta(w)) < \frac{f}{\zeta}(w)$, and of the opposite sign as that of $\lambda$ if and only if $\frac{\rho_t'}{\rho_t} > \frac{f}{\zeta}(w)$.

As noted in Section~\ref{sect:commutator}, we need $B_0$ to have order $\frac{2s - m + 1}{2}$ in both the $s < s_0$ case and the $s > s_1$ case. We then define $s_0$ so that on the support of the symbols, $s < f + \frac{m-1}{2}$. We may, of course, make the supports as small as we like, so long as $g_{2, 0}$ is nonzero on $\kappa^{-1}(q)$, and further, as the values of the symbols are irrelevant for $\zeta \leq \zeta_0$(in the sense that order $-\infty$ error terms are irrelevant), we only need this to hold for $\zeta > \zeta_0 > 0$. It is thus optimal to choose (and so we take as definition)
$$s_0 := \sup_{U_0' \subset U_0 \: \mathrm{with} \: q \in U_0', \zeta_0 > 0} \biggl( \inf_{\{w \in U | \kappa(w) \in U_0', \zeta(w) > \zeta_0\}}  f(w) + \frac{m - 1}{2}\biggr).$$

 If we assume that $\sigma_{m-1}(\frac{P - P^*}{2i})$ has homogeneous representative, then this simplifies. With that choice for $\sigma_{m-1}(\frac{P - P^*}{2i})$ in the definition of $f$, $f$ is homogeneous of degree $0$, so we may consider it a function on $S^*X$, and take
$$s_0 = f(q) + \frac{m - 1}{2}.$$

To be concrete,  we note that in the $s < s_0$ case, we take $$\rho_t(\zeta) = \zeta^\frac{2s - m + 1}{2} \hat{\chi}(t\zeta),$$ where $\hat{\chi} \in C^\infty_c(\RR)$ is identically $1$ in a neighborhood of $0$. The reader may explicitly verify that the above choice of $s_0$ works.

In the $s > s_1$ case, we must regularize so that $g_{2, t}$ has, for $t>0$, order $s_1$ because of the assumed a priori regularity $q \notin \WF^{s_1}(u)$. For $t>0$ we must then have $b_t$ of order $\frac{2s_1 - m +1}{2}$. Thus we must have $s_1 > f + \frac{m - 1}{2}$ on the supports of the symbols. As above, we may shrink the supports of the symbols, and further this only need be valid for $\zeta > \zeta_0 > 0$. It is thus optimal to choose (and so for Theorem~\ref{thm:general} we take as the defining requirement) any $s_1$ such that
\begin{equation}
s_1 > \inf_{U_0' \subset U_0, q \in U_0', \zeta_0 > 0} \biggl( \sup_{\{w \in U | \kappa(w) \in U_0', \zeta(w) > \zeta_0\}}  f(w) + \frac{m - 1}{2}\biggr). \tag{Thm~\ref{thm:general}}
\end{equation}

 If we assume that $\sigma_{m-1}(\frac{P - P^*}{2i})$ has homogeneous representative, then this again simplifies:
$$s_1  > f(q) + \frac{m - 1}{2}$$
where since $f$ is then homogeneous of degree $0$, we take it to be a function on $S^*X$. Hence, as in the statement of Theorem~\ref{thm:classical}, we may choose any $s_1 > s_0$.

To be concrete, we note that when $s > s_1$, we take $\rho_t(\zeta) = \zeta^\frac{2s - m + 1}{2} (1 + t\zeta)^{s_1 - s}$. The reader may again explicitly verify that any such above choice of $s_1$ works.


In order to prove Theorem~\ref{thm:tau}, we do not need to regularize, and we simply take the operators and symbols with $t = 0$. Thus the value of $s_0$ is the same in this case, and we can take $s_1$ to realize the lower bound for $s_1$ in Theorem~\ref{thm:tau}:
\begin{equation}
s_1 = \inf_{U_0' \subset U_0, q \in U_0', \zeta_0 > 0} \biggl( \sup_{\{w \in U | \kappa(w) \in U_0', \zeta(w) > \zeta_0\}}  f(w) + \frac{m - 1}{2}\biggr). \tag{Thm~\ref{thm:tau}}
\end{equation}
In the case where $\sigma_{m-1}(\frac{P - P^*}{2i})$ has a homogeneous choice, we can take $s_0 = s_1$. 

The above determined values of $s_0$ and $s_1$ may appear to depend on the choices of $\zeta$, representative of $\sigma_{m - 1}(\frac{P - P^*}{2i})$ (when $\sigma_{m-1}(\frac{P - P^*}{2i}) $ is not assumed to have a homogeneous choice), and density on $X$ (which determines $P^*$). The values are (as one should hope) independent of such choices.
\begin{itemize}
\item If we instead chose any other $\zeta_1: U \rightarrow \RR_+$, homogeneous of degree $1$, then we would have $\zeta_1 = g(\phi_0) \zeta$ for some $g: V_0 \rightarrow \RR_+$. As $\lambda$ depends on $\zeta$, we would  define a different 
\begin{align*}
\lambda_1 & = - H_p \zeta_1 \\ 
&  = -\zeta H_p g(\phi_0) - g(\phi_0) H_p \zeta \\
& = -\zeta H_p g(\phi_0) + g(\phi_0) \lambda
\end{align*}
As $H_p$ is radial along $\kappa^{-1}(q)$, the first term vanishes on $\kappa^{-1}(q)$, so it does not contribute in the formulas. Further, the $g(\phi_0)$ factors cancel in the fraction. Thus our formulas are independent of choice of $\lambda$.

\item That our choice of representative for $\sigma_{m-1}(\frac{P - P^*}{2i})$ does not affect the values of $s_0$ and $s_1$ is clearer: the choice is determined up to one order lower, which does not contribute in the limit $\zeta_0 \rightarrow \infty$.

\item If we chose a different density, then the adjoint operator to $P$ would be of the form $f^{-1} P^* f$, where $P^*$ is the adjoint from the original density choice, and $f \in C^\infty(X)$. We have $f^{-1} P^* f = P^* + f^{-1}[P^*, f]$. Since $H_p$ is radial at $\kappa^{-1}(q)$, $H_p f = 0$, so $\sigma_{m-1}(f^{-1}[P^*, f])$ vanishes at $\kappa^{-1}(q)$. This difference does not contribute in the formulas for $s_0$ and $s_1$. 
\end{itemize}

\section{Classical Dynamics} \label{sect:classicaldynamics}

In order to prove Theorem~\ref{thm:general}, we first must have some understanding of the symplectic geometry. First, we choose some convenient coordinates, then as a consequence we derive a geometric lemma useful for the $s < s_0$ case. From now on, we fix $P$ as in Theorem~\ref{thm:general}, and set $$\Sigma = \Sigma(P).$$

\subsection{Choice of coordinates} \label{sect:coordinates}

Let $\IL = \{f \in C^\infty(\Sigma \cap U)| \  f|_\Lambda = 0\}$, the ideal of smooth functions on $\Sigma \cap U$ which vanish on $\Lambda$, where $U$ is a conic open subset of $T^*X \backslash o$. Using the facts that $H_p$ is radial on $\Lambda$ and that $\Lambda$ is conic, we have
$$H_p: \IL \rightarrow \IL$$
and thus, as would be a consequence with any such vector field,
$$H_p: \IL^2 \rightarrow \IL^2$$
so we have the induced
$$\tilde H_p: C^\infty(\Sigma \cap U)/\IL^2 \rightarrow C^\infty(\Sigma \cap U)/\IL^2.$$
Our goal in this section is to choose coordinates which correspond to eigenvectors of this map, for a particular choice of $U$. We assume that $P \in \Psi^m(X)$ is as in the statement of Theorem~\ref{thm:general}.

\begin{lem} \label{lem:coordinates}
There exists a conic open neighborhood $U \subset T^*X\backslash o$ of $\kappa^{-1}(q)$ with coordinate chart $$\phi: U \rightarrow V \subset \RR_{\eta_0} \times \RR^{n-1}_\alpha \times \RR^{n-1}_\beta \times \RR_{+, \zeta}$$ such that $\kappa^{-1}(q) = \{\eta_0 = 0, \alpha = \beta = 0\}$, $\Lambda \cap U = \{\eta_0 = 0, \alpha = 0\}$, and $\Sigma \cap U = \{ \eta_0 = 0\}$, with $\zeta$ is homogeneous of degree $1$ and $\eta_0, \alpha, \beta$  homogeneous of degree $0$ (with respect to the $\RR_+$ action $\mu$); in addition,
\begin{align}
	\iota^* H_p \alpha_i  & \in \frac{\lambda}{\zeta}\alpha_i + \IL^2 \\
	\iota^* H_p \beta_i & \in \IL^2 \\
	H_p \zeta  & = - \lambda
\end{align}
with $\lambda \in S^m(U)$ elliptic, where $\iota: \Sigma(P) \cap U \hookrightarrow U$ is the inclusion map.
\end{lem}

\begin{rmk}
If $U'$ is any other conic open neighborhood of $\kappa^{-1}(q)$, then $U \cap U'$ also has such a coordinate chart, i.e., we can always shrink $U$, so long as it still contains $\kappa^{-1}(q)$, and it will still have the desired coordinate chart. This will be useful as we prove Theorem~\ref{thm:general}.
\end{rmk}

\begin{proof}

We start off by choosing $U$, an open conic neighborhood of $\kappa^{-1}(q)$, so that it has a canonical coordinate chart $\varphi: U \rightarrow V' \subseteq \RR^n_x \times \RR^n_\xi$, such that $\varphi(\Lambda \cap U) = V' \cap N^*\{x_n = 0\} \backslash o$ and $\kappa^{-1}(q) = \{x = 0, \xi_1 = \ldots = \xi_{n-1} = 0, \xi_n > 0\}$, where $N^*Y$ is the conormal bundle of $Y \subset X$, so in this case $\Lambda \cap U = U' \cap \{x_n = 0, \xi_1 = \ldots = \xi_{n-1} = 0\}$. This choice can be made: see, for instance, \cite{hormander3}, Theorem 21.2.8. We shrink $U$ so that $\xi_n > 0$ on $U$. 

Define an intermediate coordinate chart $$\phi_1: V' \rightarrow V'' \subseteq \RR^{n-1}_y \times \RR_z \times \RR^{n-1}_\theta \times \RR_\zeta$$ by
\begin{align*}
y_i & = x_i, i<n \\
z & = x_n \\
\theta_i & = \frac{\xi_i}{\xi_n}, i  < n\\
\zeta & = \xi_n,
\end{align*}
so $y, z, \theta$ are homogeneous of degree $0$, and $\zeta$ is homogeneous of degree 1. 

In what follows we sometimes write $p$ for $p \circ \varphi^{-1} \circ \phi_1^{-1}$, and similarly for other functions, in order to make formulas less cluttered. We have 
\begin{align*}
\phi_{1*} \partial_{\xi_i} & = \frac{1}{\zeta} \partial_{\theta_i}, i < n \\
\phi_{1*} \partial_{\xi_n} & = \partial_\zeta - \sum_i \frac{\theta_i}{\zeta} \partial_{\theta_i} \\
(\phi_1^{-1})^* d\xi_i & = \zeta d\theta_i + \theta_i d\zeta \\
(\phi_1^{-1})^* d\xi_n & = d\zeta
\end{align*}
and thus 
$$\phi_{1*} \varphi_* H_p = \partial_\zeta p \partial_z + \frac{1}{\zeta} \biggr(\sum^{n - 1}_{i = 1} (\theta_i \partial_z p - \partial_{y_i}p)\partial_{\theta_i} + \partial_{\theta_i} p( \partial_{y_i} - \theta_i \partial_z) \biggl) - (\partial_z p) \partial_\zeta.$$
If we let $\omega$ be the standard symplectic form on $T^*\RR^n$, we have 
$$(\varphi^{-1} \circ \phi_1^{-1})^* \omega = dz \wedge d\zeta + \sum^{n-1}_{i=1} dy_i \wedge (\zeta d\theta_i + \theta_i d\zeta).$$

Noting that $\phi_1(\Lambda) = \{z = 0, \theta = 0\}$, $p|_\Lambda = 0$ implies that $\partial_{y_i} p = \partial_\zeta p = 0$ on $\phi_1(\varphi(\Lambda))$. In order that $H_p$ be radial on $\Lambda$, we must have that $\partial_{\theta_i} p = 0$ on $\phi_1(\varphi(\Lambda))$ as well. In order for nondegeneracy $dp \neq 0$ to hold, we must have $\partial_z p \neq 0$ on $\Lambda$. After potentially shrinking $U$ further (and so also shrinking $V'$ and $V''$), there is, by the implicit function theorem, an 
$$f: \{(y, \theta, \zeta)\: | \: \exists \: z \: \mathrm{with} \: (y, z, \theta, \zeta) \in V''\} \rightarrow \RR$$
such that 
$$p(\phi_1^{-1}(y, f(y, \theta, \zeta), \theta, \zeta)) = 0$$
and $f(0, 0, 1) = 0$.
As $p$ is homogeneous, we have $\partial_\zeta f = 0$, and using the above conditions on $p$ at $\Lambda$, we have $\partial_{y_i} f = \partial_{\theta_i} f = 0$ on $\phi_1(\varphi(\Lambda))$, for all $i$, so in particular $f(y, 0, \zeta) = 0$. As this implies that $\partial_{\theta_i} \partial_{y_j} f = \partial_{y_i} \partial_{y_j} f = 0$, $f(y, \theta, \sigma) \in \IL^2$ (considered a function on $\Sigma \cap U$ because $y, \theta, \zeta$ are coordinates for $\Sigma \cap U$). Thus we have the following:
\begin{align*}
\iota^* \varphi^* \phi_1^* \partial_{y_i} p & \in \IL^2 \\
\iota^* \varphi^* \phi_1^* \partial_{\theta_i} p & \in \IL
\end{align*}

We choose $\alpha_i = \varphi^* \theta_i$ and $\eta_0 = \frac{p}{\varphi^* \zeta^m}$. To finish the lemma, it suffices to choose $\beta_i(y, \theta)$ with $\partial_{y_j} \beta_i = \delta_{ij}$ on $\Lambda$ and 
$$\iota^* \varphi^* \phi_1^* \left( \sum_j \theta_j \partial_z p \: \partial_{\theta_j} \beta_i + \partial_{\theta_j}p \: \partial_{y_j} \beta_i \right) \in \IL^2.$$
 This is easy to accomplish: we can, for instance, let 
$$\beta_i = \varphi^* (y_i - \frac{\partial_{\theta_i} p}{\partial_z p}(y, f(y, \theta), \theta))$$
where we above omit dependence on $\zeta$, since $\frac{\partial_{\theta_i} p}{\partial_z p}$ is homogeneous of degree $0$. Lastly, in order to assure that $\lambda = \varphi^* \phi_1^* (\partial_z p)$ is elliptic, we may need to shrink $U$.
\end{proof}

\subsubsection{The compactified cotangent bundle picture, continued} \label{sect:compactifiedcontinued}

This section is optional and meant to give a nice picture of the dynamics involved. Let $\overline{\Sigma} = \Sigma \cup \hS \subset T^*X \cup S^*X = \overline{T}^*X$, and $\overline{\Lambda} = \Lambda \cup L \subset \overline{T}^*X$. Then given coordinates as in Lemma~\ref{lem:coordinates}, we define $x$, a boundary defining function defined on $\tilde{U} = U \cup \kappa(U) \subset \overline{T}^*X$, as in Section~\ref{sect:compactified}: $x = \frac{1}{\zeta}$ on $U$, and $x = 0$ on the boundary $\kappa(U)$. Further, $\eta_0, \alpha,$ and $\beta$ extend to $\tilde{U}$, and together with $x$ give a coordinate chart for $\tilde U$. $W_p = x^{m-1}H_p$ extends to a vector field on $\tilde U$, tangent to the boundary (that is, $W_p \in \mathcal V_b(\tilde U)$), and for this section we take $W_p$ to be this extension. 

The eigenvalue $\lambda_0$ mentioned in Sections~\ref{sect:cosphere} and \ref{sect:compactified} is the value of $x^m \lambda$ at $q$ (note that $\frac{\lambda}{\zeta^m}$ is homogeneous of order $0$, so it extends to $\tilde U$), so here we define $\lambda_0 = x^m \lambda \in C^\infty(\tilde U)$. If we set $\mathcal I_{\overline \Lambda, \tilde U} = \{f \in C^\infty(\overline \Sigma \cap \tilde U) | \ f|_L = 0\}$ and $\iota_{\overline \Sigma}: \overline \Sigma \hookrightarrow \overline T^*X$ inclusion, then
\begin{align*}
\iota_{\overline \Sigma}^* W_p \alpha_i & \in \lambda_0 \alpha_i + \mathcal I_{\overline{\Lambda}, \tilde{U}}^2 \\
\iota_{\overline \Sigma}^* W_p \beta_i & \in \mathcal I_{\overline{\Lambda}, \tilde{U}}^2 \\
W_p x  & = \lambda_0 x.
\end{align*}
In particular, the linearization of $W_p|_{\overline \Sigma}$ at $q$ has two eigenvalues, $\lambda_0(q)$ (of multiplicity $n-1$) and $0$ (of multiplicity $n-1$). Thus we see the sink/source behavior at $L$. 

\begin{figure}
\center
\def\svgwidth{\columnwidth}
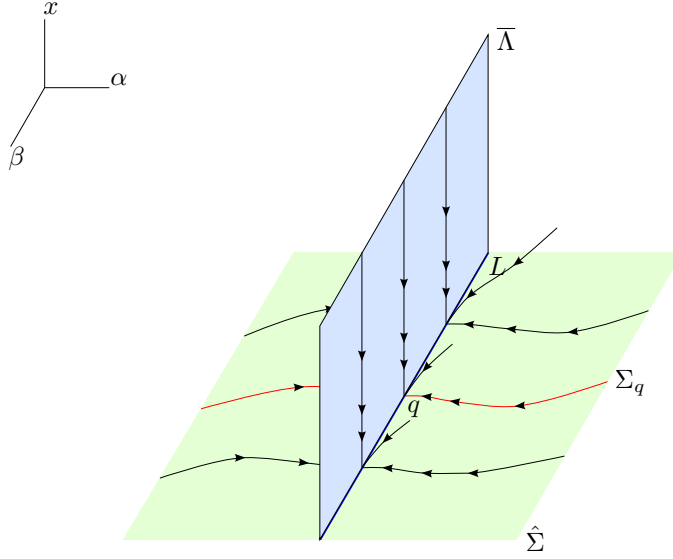
\caption{$\overline{\Sigma}$, in the case where $L$ is a sink, using the coordinates $x, \alpha, \beta$.}
\label{fig:compactified}
\end{figure}

\subsection{Geometric Lemma} \label{sect:geometriclemma}

We now state and prove a lemma which takes the regularity assumed on $u$ in the $s < s_0$ case in the statement of Theorem~\ref{thm:general}, and gives us regularity in an open subset of $S^*X$. This essentially depends on the fact that the flow lines of $\kappa_* W_p$ are well-behaved close to $L$. As before, we take $\hS = \kappa(\Sigma)$. We take $P$ as in the statement of Theorem~\ref{thm:general}, and $Q_\tau$ as in the statement of Theorem~\ref{thm:tau}. 

\begin{lem} \label{lem:geometriclemma}
Given an open neighborhood $W \subset \hS$ of $\Gamma_q \cap U_0$ for some open neighborhood $U_0 \subset S^*X$ of $q$, there is an open neighborhood $W' \subset \hS$ of $q$ such that $W' \backslash L \subseteq \{\exp(t \kappa_* W_p)w \: | \ w \in W, t \geq 0\}$ if $L$ is a sink for $\kappa_* W_p$ (respectively, $t \leq 0$ if $L$ is a source for $\kappa_* W_p$).
\end{lem}

\begin{proof}
We first shrink $U_0$ so that $U_0 = \kappa(U)$ with $U$ as in Lemma~\ref{lem:coordinates}. Using the coordinates given by Lemma~\ref{lem:coordinates}, we have a coordinate chart $$\psi: U_0 \cap \hS \rightarrow V_{\hS} \subseteq \RR^{n-1}_\alpha \times \RR^{n-1}_\beta.$$
Here $$\psi_{*} (\kappa_*W_p)_{\hS} = \sum_i \lambda(\alpha, \beta) (\alpha_i + w_i(\alpha, \beta))\partial_{\alpha_i} + r_i(\alpha, \beta)\partial_{\beta_i}$$
where $w_i, r_i \in \mathcal I_{L, U_0}$ (where we define, analogously, $\mathcal I_{L, U_0} = \{f \in C^\infty(\hS \cap U_0) \: | \  f|_L = 0\}$). To analyze this, we introduce a blow up of $L \cap U_0$ with blowdown map $$B: [U_0 \cap \hS ; U_0 \cap L] \rightarrow U_0 \cap \hS.$$
This can easily be described in terms of coordinates: $[U_0 \cap \hS ; U_0 \cap L]$ is diffeomorphic to a neighborhood of $\{r = 0, \beta = 0\}$ in $\RR_{+,r} \times \mathbb S^{n-2}_\omega \times \RR^{n-1}_\beta$. In these coordinates and the coordinates $(\alpha, \beta)$ for $\hS \cap U_0$, $B$ is the map $(r, \omega, \beta) \mapsto (r \omega, \beta)$. We then have $r$ as a boundary defining function for $B^{-1}(L)$.

\begin{figure}
\center
\def\svgwidth{\columnwidth}
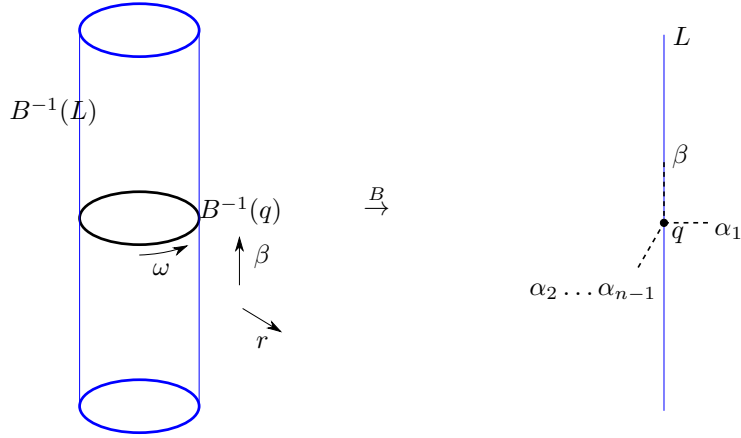
\caption{The blow up construction}
\label{fig:blowup}
\end{figure}

$\kappa_*W_p|_{\hS}$ then lifts uniquely to a vector field on $[U_0 \cap \hS ; U_0 \cap L]$, and in these coordinates, it is of the form
$$(\lambda(r, \omega, \beta) r + w(r, \omega, \beta))\partial_r + w_i(r, \omega, \beta)\partial_{\omega_i} + r_i(r, \omega, \beta)\partial_{\alpha_i}$$
where $w, w_i, r_i \in \mathcal I_L^2$. This is of the form $$r V_\perp$$ where $V_\perp$ is transverse to $B^{-1}(L)$. The lifts of the integral curves of $\kappa_*W_p|_{\hS}$ are the same as the flow lines of $V_\perp$ away from $B^{-1}(L)$, so to prove the lemma we may simply study the latter. 

\begin{figure}
\center
\def\svgwidth{\columnwidth}
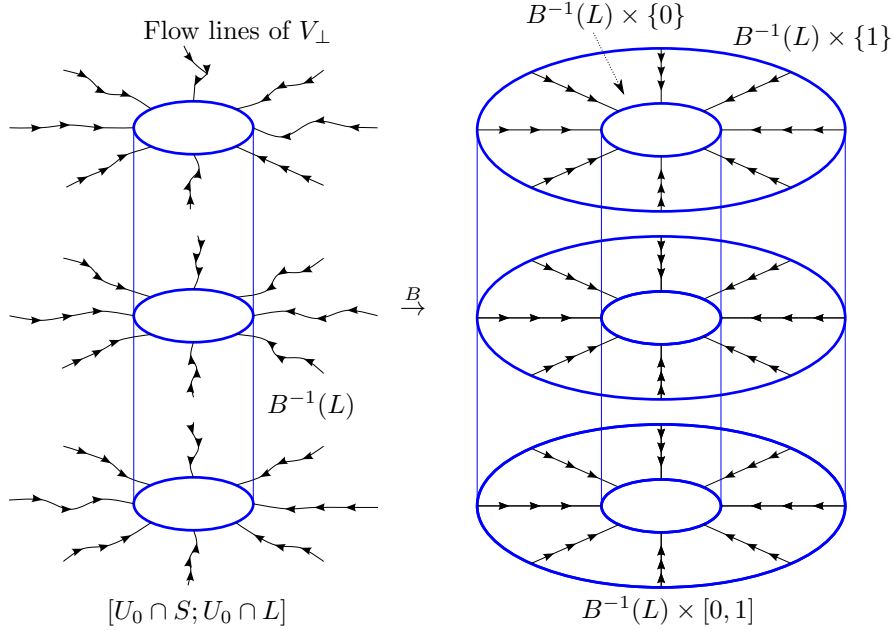
\caption{The map $\psi$}
\label{fig:straightenedblowup}
\end{figure}

As is standard ODEs (see, for instance, Chapter 1 of \cite{taylor}), the flow of $V_\perp$ gives a diffeomorphism $$\varphi: W'' \subset [U_0 \cap \hS ; U_0 \cap L] \rightarrow B^{-1}(L) \times [0, 1]$$ which extends  the `identity' $B^{-1}(L) \rightarrow B^{-1}(L) \times \{0\}$ (and $W''$ is an open neighborhood of $B^{-1}(L)$). 

The set $W$ in the assumption of the lemma gives an open set $U' \subset B^{-1}(L)$ of $B^{-1}(q)$ such that $$U' \times \{1\} \subset \varphi(B^{-1}(W)).$$ We can then take $W' = B(\varphi^{-1}(U' \times [0, 1)))$.
\end{proof}

\begin{cor} \label{cor:geometriccor}
If $\WF^s(u) \cap U_0 \cap \Gamma_q = \emptyset$ for some open neighborhood $U_0 \subset S^*X$ of $q$ with $\WF^{s - m + 1}(Pu) \cap U_0 = \emptyset$, then $\WF^s(u) \cap (W' \backslash  L) = \emptyset$ for some open neighborhood $W' \subset \hS$ of $q$.
\end{cor}

\begin{proof}
Since $WF^s(u)$ is a closed set, $U_0 \backslash \WF^s(u)$ is such a $W$ as in the statement of Lemma~\ref{lem:geometriclemma}. The result then follows from Lemma~\ref{lem:geometriclemma} and Theorem~\ref{thm:hormander}.
\end{proof}

\begin{cor} \label{cor:geometriccortau}
If $L$ is a sink for $\kappa_* W_p$ and $\WF_{L^\infty([0, 1])}(u_\tau) \cap U_0 \cap \Gamma_q = \emptyset$ for some open neighborhood $U_0 \subset S^*X$ of $q$ with $\WF^{s - m + 1}_{L^\infty([0, 1])}((P - i Q_\tau)u_\tau) \cap U_0 = \emptyset$, then $\WF^s_{L^\infty([0, 1])}(u_\tau) \cap (W' \backslash L) = \emptyset$ for some open neighborhood $W' \subset \hS$ of $q$.
\end{cor}

\begin{proof}
This follows in the same way as the above corollary, this time applying Lemma~\ref{lem:geometriclemma} and Theorem~\ref{thm:melrosetau}.
\end{proof}

\section{Commutator Argument} \label{sect:commutator}

In this section, we state the operators which we will construct in Section~\ref{sect:operatorconstruction}, and then assuming their construction, prove Theorem~\ref{thm:general}. First, we need a general lemma regarding families of pseudodifferential operators. This will help when regularizing.

\begin{lem} \label{lem:regularize}
If $A_t \in L^\infty([0, 1]_t, \Psi^r(X))$ for any $r \in \RR$, with $A_t \rightarrow A_0$ in the topology of $\Psi^{r + \delta}(X)$ for some $\delta > 0$, then $A_t \rightarrow A_0$ in the strong operator topology of operators $H^{s}(X) \rightarrow H^{s - r}(X)$, for all $s \in \RR$, for any density choice for $X$.
\end{lem}

\begin{proof}
If $v \in H^{s + \delta}(X)$, then given the continuity assumption  $A_t \rightarrow A_0$ and the fact that the standard map $\Psi^{r + \delta}(X) \rightarrow \mathcal L(H^{s + \delta}(X), H^{s - r}(X))$ is continuous, we have that $A_t v \rightarrow A_0 v$ in the topology of $H^{s - r}$. The assumption $A_t \in L^\infty([0, 1]_t, \Psi^r(X))$ implies that, if $u \in H^s(X)$, then $A_t u$ is bounded in $H^{s - r}(X)$. As $H^{s + \delta}(X)$ is dense in $H^s(X)$, $A_t u \rightarrow A_0 u$ in $H^{s - r}(X)$. 
\end{proof}

\subsection{$s < s_0$ case}

\begin{lem} \label{lem:s0operators}
Given an open neighborhood $U_0 \subset S^*X$ of $q$, there exist
\begin{align*}
 	B = (B_t)_{t \in [0, 1]} & \in L^\infty([0,1]_t, \Psi^{\frac{2s + m - 1}{2}}(X)), \\
	G_1 = (G_{1, t})_{t \in [0, 1]}, G_2 = (G_{2, t})_{t \in [0, 1]} & \in L^\infty([0,1]_t, \Psi^{s}(X)), \\
	E = (E_t)_{t \in [0, 1]} & \in L^\infty([0, 1]_t, \Psi^{2s}(X)), \\ 
	F = (F_t)_{t \in [0, 1]} & \in L^\infty([0,1]_t, \Psi^{2s - 1}(X)), \\
	H = (H_t)_{t \in [0, 1]} & \in L^\infty([0, 1]_t, \Psi^{s - m + 1}(X)), \\
	J = (J_t)_{t \in [0, 1]} & \in L^\infty([0, 1]_t, \Psi^{2s - m}(X)),
\end{align*}
such that
\begin{align*}
\frac{B_t^2 P - P^* B_t^2}{2i} & = \sgn(\lambda) (G_{1, t}^* G_{1,t} + G_{2, t}^*G_{2, t}) + E_t + F_t \\
B_t^2 & = G_{2, t} H_t + J_t
\end{align*}
and
\begin{enumerate}
	\item all operators are in $\Psi^{-\infty}(X)$ for $t > 0$,
	\item $B_t, G_{j, t}$ are continuous in the topologies of $\Psi^{\frac{2s - m + 1}{2} + \delta}(X), \Psi^{s + \delta}(X)$, respectively, for all $\delta > 0$,
	\item all operators have $\WF'_{L^\infty([0, 1])}$ contained in $U_0$,
	\item $\WF'_{L^\infty([0,1])}(E_t) \cap L = \emptyset$,
	\item $B_t^* = B_t$ (assuming a choice of density for $X$),
	\item $q \in \mathrm{Ell}(G_{2, 0})$.
\end{enumerate}
\end{lem}

\begin{rmk}
More is true: we can actually take $F_t, J_t \in L^\infty([0, 1]_t, \Psi^{-\infty})$. This is not needed in this proof of Theorem~\ref{thm:general}, but we prove an analogue in Section~\ref{sect:tausource} which carries over.
\end{rmk}

For now, we assume this lemma and proceed to prove the $s < s_0$ case of Theorem~\ref{thm:general}. 

\begin{proof} ($s < s_0$ case of Theorem~\ref{thm:general})

We may assume, by shrinking $U_0$ if necessary, the following:
\begin{itemize}
	\item $\WF^{s - \frac{1}{2}}(u) \cap U_0 = \emptyset$, as $q \notin \WF^{s'}(X)$ for some $s'$ and we can inductively improve regularity by $\frac{1}{2}$, each time making $U_0$ smaller.
	\item $\WF^s(u) \cap \hS \cap (U_0 \backslash L) = \emptyset$, by Corollary~\ref{cor:geometriccor}
	\item $\WF^{s - m + 1}(Pu) \cap U_0= \emptyset$.
\end{itemize}

As in the sketch of the proof, we begin by choosing a density for $X$, which gives us distributional pairings. In order to avoid some complications with pairings, we if necessary modify the constructed operators to have compactly supported Schwartz kernels. For $t > 0$, the following pairings are well-defined, and equality holds:
\begin{align*}
\frac{1}{2i}\lrangle{u, (B_t^2 P - P^* B^2_t) u} = & \sgn(\lambda) (\lrangle{u,G_{1, t}^* G_{1,t} u} + \lrangle{u, G_{2, t}^* G_{2, t} u}) \\ 
& + \lrangle{u, E_t u} + \lrangle{u, F_t u}.
\end{align*}
We have $\lrangle{u, G_{j, t}^2 u} = \|G_{j, t} u\|^2$, and on the left-hand side, 
\begin{align*}
	|\frac{1}{2i}\lrangle{u, B_t^2 P - P^* B_t^2 u}| & = |\frac{1}{2i}\bigr(\lrangle{u, B_t^2 P u} - \lrangle{B_t^2 P u, u})| \\
	& = |\mathrm{Im} \lrangle{u, B_t^2 P u}| \\
	& = |\mathrm{Im} (\lrangle{u, G_{2,t} H_t P u} + \lrangle{u, J_t P u})| \\
	& = |\mathrm{Im} (\lrangle{G_{2,t} u, H_t P u} + \lrangle{u, J_t P u})| \\
	& \leq |Im \lrangle{u, J_t Pu}| +  \|G_{2,t} u\|\|H_t P u\| \\
	& \leq |Im \lrangle{u, J_t Pu}| + \frac{c}{2}\|G_{2,t} u\|^2 + \frac{1}{2c}\|H_tPu\|^2
\end{align*}
for any $c > 0$, which we choose to be $<2$. We then have
\begin{align*}
\|G_{1, t}\|^2 + (1 - \frac{c}{2})\|G_{2, t}u\|^2 & \leq \frac{1}{2c}\|H_tPu\|^2 + |\mathrm{Im}\lrangle{ u, J_t Pu}| \\
& \quad + |\lrangle{u, E_t u}| + |\lrangle{u, F_t u}|
\end{align*} 
By the assumed regularity of $Pu$, $\|H_t Pu\|$ and $\lrangle{u, J_t Pu}$ remain bounded as $t \rightarrow 0$. Since $\WF'_{L^\infty([0,1])}(E_t) \cap \WF^s(u) = \emptyset$ (away from $\hS$, too, by elliptic regularity), $\lrangle{u, E_t u}$ remains bounded as $t \rightarrow 0$. Lastly, by assumption on the regularity of $u$ in $\kappa(U)$, $\lrangle{u, F_t u}$ remains bounded. Thus $G_{1,t} u$ and $G_{2, t} u$ remain bounded in $L^2(X)$. By Banach-Alaoglu, $G_{2, t} u$ has a weakly convergent sequence $G_{2, t_n} u$ in $L^2(X)$. On the other hand, by the continuity assumption on $G_{2, t}$, $G_{2, t} u \rightarrow G_{j, 0} u$ in the sense of distributions. Thus $G_{2, t_n} u \rightarrow G_{2, 0} u$ in $L^2(X)$, so $G_{2, 0} u \in L^2(X)$. Thus $\mathrm{Ell}(G_{2, 0}) \cap \WF^s(u) = \emptyset$, so $q \notin \WF^s(u)$.
\end{proof}

\subsection{$s > s_1$ case}

\begin{lem} \label{lem:s1operators}
Given an open neighborhood $U_0 \subset S^*X$ of $q$, there exist 
\begin{align*}
B = (B_t)_{t \in [0, 1]} & \in L^\infty([0,1]_t, \Psi^{\frac{2s + m - 1}{2}}(X)), \\
G_1 = (G_{1, t})_{t \in [0, 1]}, G_2 = (G_{2, t})_{t \in [0, 1]} & \in L^\infty([0,1]_t, \Psi^{s}(X)), \\
E = (E_t)_{t \in [0, 1]} & \in L^\infty([0, 1]_t, \Psi^{2s}(X)), \\
F = (F_t)_{t \in [0, 1]} & \in L^\infty([0,1]_t, \Psi^{2s - 1}(X)), \\
H = (H_t)_{t \in [0, 1]} & \in L^\infty([0,1]_t, \Psi^{s - m + 1}(X)), \\
J = (J_t)_{t \in [0, 1]} & \in L^\infty([0,1]_t, \Psi^{2s - m}(X),
\end{align*}
such that
\begin{align*}
\frac{B_t^2 P - P^* B_t^2}{2i} & = -\sgn(\lambda) (G_{1, t}^* G_{1, t}  + G_{2, t}^*G_{2, t}) + E_t + F_t \\
B_t^2 & = G_{2, t} H_t + J_t
\end{align*}
with
\begin{enumerate}
	\item for $t > 0$, $B_t \in \Psi^\frac{2s_1 - m + 1}{2}(X), G_{j, t} \in \Psi^{s_1}(X), E_t \in \Psi^{2s_1}(X)$, \\ $F_t \in \Psi^{2s_1 - 1}(X), H_t \in \Psi^{s_1 - m + 1}, J_t \in \Psi^{2s_1 - m}(X)$,
	\item $B_{t}, G_{j,t}$ are continuous in the topologies of $\Psi^{\frac{2s + m - 1}{2} + \delta}(X), \Psi^{s + \delta}(X)$, respectively, for all $\delta > 0$,
	\item all operators have $\WF'_{L^\infty([0, 1])}$ contained in $U_0$,
	\item $\WF'_{L^\infty([0,1])}(E_t) \cap \hS = \emptyset$,
	\item $B_{t}^* = B_{t}$ (assuming a choice of density for $X$),
	\item $q \in \mathrm{Ell}^s(G_{2,0})$.
\end{enumerate}
\end{lem}

\begin{rmk}
As with Lemma~\ref{lem:s0operators}, we can actually take $F_t, J_t \in L^\infty([0, 1]_t, \Psi^{-\infty}(X))$.
\end{rmk}

As above, we assume this lemma is true and proceed to prove the rest of Theorem~\ref{thm:general}.

\begin{proof} ($s > s_1$ case of Theorem~\ref{thm:general})

We may assume, by shrinking $U_0$ if necessary, the following:
\begin{itemize}
	\item $\WF^{s - \frac{1}{2}(u)}(X) \cap U_0 = \emptyset$, as $q \notin \WF^{s_1}(X)$, and we can inductively improve regularity by $\frac{1}{2}$, each time making $U_0$ smaller.
	\item $\WF^{s_1}(u) \cap U_0 = \emptyset$.
	\item $\WF^{s - m + 1}(Pu) \cap U_0= \emptyset$.
\end{itemize}
As before, we choose a density for $X$, which gives us distributional pairings, and again we may take the constructed operators to have compactly supported Schwartz kernels. For $t > 0$, the following pairings are well-defined (here we use $\WF^{s_1}(u) \cap U_0 = \emptyset$), and equality holds:
\begin{align*}
\frac{1}{2i}\lrangle{u, (B_t^2 P - P^* B^2_t) u} = & - \sgn(\lambda) (\lrangle{u, G_{1, t}^* G_{1 ,t}u} + \lrangle{u, G_{2, t}^* G_{2, t}u}) \\
& +  \lrangle{u, E_t u} + \lrangle{u, F_t u},
\end{align*}

To deal with the left-hand side, we need a lemma:
\begin{lem} \label{lem:integratebyparts}
For $t>0$, $$\lrangle{u, (B_t^2 P - P^* B_t^2) u} = \lrangle{u, B_t^2  Pu} - \lrangle{B_t^2 Pu, u}.$$
\end{lem}

\begin{proof}
It is tempting to simply conclude this immediately, but note that it is not clear just by the regularity assumptions that $\lrangle{u, P^* B_t^2 u}$ is well-defined. This was not a problem in the $s < s_0$ setting because there $B_t \in \Psi^{-\infty}(X)$ for $t >  0$, but now the order is higher. Thus to prove this, we regularize again. This is a fairly standard argument, but since there are several details that need to be verified in order to be sure that it works in this instance, we write the argument out in some detail. Let $A_{t'} \in L^\infty([0, 1]_{t'}, \Psi^0(X))$ be such that $A_{t'} \in \Psi^{-\infty}(X)$ for $t' > 0$, and $A_{t'} \rightarrow \Id$ as ${t'} \rightarrow 0$ in $\Psi^\delta(X)$ for $\delta > 0$.

Fixing $t > 0$, then for $t' > 0$, we have
\begin{align*}
\lrangle{u, A_{t'} (B_t^2 P - P^* B_t^2) u} & = \lrangle{u, A_{t'} B_t^2 Pu} - \lrangle{u, P^* B_t^2 A_{t'} u} \\ 
& \quad + \lrangle{u, [P^* B_t^2, A_{t'}]u} \\
& = \lrangle{u, A_{t'} B_t^2 Pu}  - \lrangle{A_{t'} B_t^2 P u, u} \\
& \quad + \lrangle{u, [P^*B_t^2, A_{t'}]u}.
\end{align*}
Note that, as $t' \rightarrow 0$, $[P^*B_t^2, A_{t'}] \rightarrow 0$ in $\Psi^{2s_1 + \delta}(X)$ for $\delta > 0$. 

Let $A' \in \Psi^0(X)$ be such that $\WF'(A') \subset U_0$ and $\WF'(\operatorname{Id} - A') \cap \WF'(B_t) = \emptyset$. Then 
\begin{align*}
\lrangle{u, A_{t'} B_t^2 P u} = & \lrangle{A' u, A_{t'} A' B_t^2 P u}  +  \lrangle{u, (\Id - A'^*) A_{t'} A' B_t^2 P u} \\
&+ \lrangle{u, A_{t'} (\Id - A') B_t^2 P u}.
\end{align*}
We have $A'B_t^2 Pu \in H^{s_1}(X)$, $(\Id - A'^*) A_{t'} A' B_t^2 P \in L^\infty([0, 1]_{t'}, \Psi^{-\infty}(X))$, and $A_{t'} (\Id - A') B_t^2 P \in L^\infty([0, 1]_{t'}, \Psi^{-\infty}(X))$. Thus, we can apply Lemma~\ref{lem:regularize}, and obtain $$\lrangle{u, A_{t'} B_t^2 Pu} \rightarrow \lrangle{u, A_{t'} B_t^2 Pu}$$ as $t' \rightarrow 0$. Handling the other terms similarly, we have
\begin{align*}
\lrangle{u, A_{t'} (B_t^2 P - P^* B_t^2) u} & \rightarrow \lrangle{u, (B_t^2 P - P^* B_t^2) u}, \\
\lrangle{A_{t'} B_t^2 P u, u} & \rightarrow \lrangle{B_t^2 P u, u} \\
\lrangle{u, [P^* B_t^2, A_{t'}]u} & \rightarrow 0
\end{align*}
as $t' \rightarrow 0$. This proves the lemma.
\end{proof}

Finishing the proof of Theorem~\ref{thm:general}, we have, as in the $s < s_0$ case, $$|\mathrm{Im}(\lrangle{B_t^2 u, Pu})| \leq |\mathrm{Im} \lrangle{u, J_t Pu}| + \frac{c}{2}\|G_{2,t} u\|^2 + \frac{1}{2c}\|H_t Pu\|^2,$$ for any $c > 0$, which we again take to be $<2$. We then have, for $t > 0$, 
\begin{align*}
\|G_{1, t}\|^2 + (1 - \frac{c}{2})\|G_{2, t}u\|^2 & \leq \frac{1}{2c}\|H_tPu\|^2 + |\mathrm{Im}\lrangle{ u, J_t Pu}| \\
& \quad + |\lrangle{u, E_t u}| + |\lrangle{u, F_t u}|
\end{align*}
All terms on the right side remain bounded as $t \rightarrow 0$ (the only difference from the $s < s_0$ case is that $\WF'_{L^\infty([0,1])}(E_t) \cap \hS = \emptyset$, so $\lrangle{u, E_t u}$ remains bounded simply by elliptic regularity). As in the $s < s_0$ case, we conclude that $G_{2, 0}u \in L^2(X)$, so $q \notin \WF^s(u)$. 
\end{proof}

\section{Construction of Operators} \label{sect:operatorconstruction}

Here we prove Lemmas~\ref{lem:s0operators} and \ref{lem:s1operators}. To do this, we construct symbols supported in $U = \kappa^{-1}(U_0)$, and quantize these. For this section, we do not need to be too careful about our choice of quantization. We require that our quantization $q$ satisfies $$\WF'_{L^\infty([0, 1])}(q(a_t)) = \mathrm{esssup}_{L^\infty[0, 1]}(a_t),$$ where $a_t \in L^\infty([0, 1], S^r(X))$. We also require that if $a \in S^r(T^*X)$ is real-valued, $q(a) - q(a)^* \in \Psi^{r-1}(X)$.  These are both easy to accomplish: the standard left and Weyl quantizations in $\RR^n$ satisfy this, and we can simply patch either of these together.

\subsection{Proof of Lemma~\ref{lem:s0operators}} \label{sect:s0operatorconstruction}

It suffices to produce symbols
\begin{align*}
b =  (b_t)_{t \in [0, 1]} &  \in L^\infty([0, 1]_t, S^\frac{2s - m + 1}{2}(T^*X)), \\
g_1 = (g_{1, t})_{t \in [0, 1]}, g_2 = (g_{2, t})_{t \in [0, 1]} & \in L^\infty([0,1]_t, S^s(T^*X)), \\
e = (e_t)_{t \in [0, 1]} & \in L^\infty([0,1]_t, S^{2s}(T^*X)), \\
h = (h_t)_{t \in [0, 1]} & \in L^\infty([0,1]_t, S^{s - m + 1}(T^*X)), \\
\end{align*}
such that
\begin{align*}
\frac{1}{2}H_p b_t^2 + \sigma_{m-1}(\frac{P - P^*}{2i})b_t^2 & = \sgn(\lambda)(g_{1,t}^2 + g_{2, t}^2) + e_t \\
b_t^2 & = g_{2, t} h_t
\end{align*}
with:
\begin{enumerate}
	\item all symbols of order $-\infty$ for $t > 0$,
	\item $b_t, g_{j, t}$ are continuous in the topologies of $S^\frac{2s - m + 1 + \delta}{2}(T^*X)$ and $S^{s + \delta}$, respectively, for all $\delta > 0$,
	\item $\supp(b_t), \supp(e_t), \supp(g_{j,t}) \subset \kappa^{-1}(U_0)$,
	\item $\mathrm{esssup}_{L^\infty([0, 1])}(e_t) \cap \Lambda = \emptyset$,
	\item all symbols real-valued,
	\item $q \in \mathrm{Ell}(g_{2, t})$.
\end{enumerate}
Indeed, let $B_t = \frac{q(b_t) + q(b_t)^*}{2}$, $G_{j, t} = q(g_{j,t})$, $E_t = q(e_t)$ and $H_t = q(h_t)$. Then $\sigma_{2s}(B^2 P - P^* B^2 - \sgn(\lambda)( G_{1,t}^* G_{1,t} + G_{2, t}^*G_{2, t}) - E_t) = 0$, so the error $F_t$ is as desired. Further, we have $$B_t^2 = H_t G_{2,t }+ J_t$$ for some $J_t$ as desired.

To construct $b_t$, we first assume (by shrinking $U := \kappa^{-1}(U_0)$ if necessary) that $U$ has a coordinate chart $\phi$ as in Lemma~\ref{lem:coordinates}.
We choose functions $\chi_0, \chi_1, \chi_2 \in C^\infty(U)$ homogeneous of degree $0$, and $\rho_t \in L^\infty([0, 1]_t, S^\frac{2s - m + 1}{2}(X))$, so that $\chi_0 \chi_1 \chi_1$ functions as the cutoff $\chi(\phi_0)$ did in Section~\ref{sect:sketch}, and $\rho_t$ is the weight with desired order properties. As in Section~\ref{sect:s0s1formulas}, we let $\hat{\chi} \in C^\infty_c(\RR)$ be identically $1$ in a neighborhood of $0$. Then let $\rho_t = \zeta^\frac{2s - m + 1}{2} \hat{\chi}(t \zeta)$. As in our definition of $s_0$, choose an open neighborhood $U_0' \subseteq U_0$ of $q$, along with $\zeta_0 \in \RR_+$, so that $$\rho_t H_p \rho_t + \sigma_{m-1}\bigl(\frac{P - P^*}{2i}\bigr) \rho_t^2$$ remains the same sign as $\lambda$ inside $\kappa^{-1}(U_0') \cap \zeta^{-1}((\zeta_0, \infty))$. As this is only true for $\zeta > \zeta_0$, we need to include an additional cutoff (this also serves to make homogeneous symbols smooth up to the zero-section of $T^*X$) $\hat{\rho}: U \rightarrow \RR$ such that $\hat{\rho}$ is identically $0$ for $\zeta \leq \zeta_0$ and identically $1$ for $\zeta \geq \zeta_0 + 1$. We then let $$b_t =  \hat{\rho}(\zeta) \chi_0 \chi_1 \chi_2 \rho_t$$ inside $U$ and identically $0$ outside $U$. This will have the desired properties if:
\begin{itemize}
	\item $\sqrt{\sgn(\lambda) \chi_1 H_p \chi_1 }$ is real-valued and smooth,
	\item $\kappa(\supp(\chi_0 \chi_1 \chi_2))$ is a compact subset of $U_0'$. 
	\item $\supp(\chi_0 \chi_1 H_p \chi_2) \cap \Lambda = \emptyset$, and
	\item $\supp \chi_1 \chi_2 H_p \chi_0 \cap \Sigma = \emptyset$,
\end{itemize}

To construct $\chi_0, \chi_1$, and $\chi_2$, let $\eta_1, \eta_2: U \rightarrow \RR$ be defined by $\eta_1 = |\beta|^2 + C|\alpha|^2$, $\eta_2 = |\alpha|^2$, with $C < 0$ to be chosen. Recall that we define $\eta_0 = \frac{p}{\zeta^m}$ a coordinate of $\phi$ in Lemma~\ref{lem:coordinates}. Let $\tilde \chi \in C^\infty(\RR)$ so that
\begin{itemize}
	\item $\tilde \chi \geq 0$,
	\item $\tilde \chi = 1$ for $t \in (-\infty, \epsilon)$,
	\item $\tilde \chi(t) = 0$ for $t \geq T$,
	\item $\tilde \chi' \leq 0$,
	\item $\sqrt{-\tilde \chi \tilde \chi'} \in C^\infty(\RR)$,
\end{itemize}
with $T$ to be chosen, and $0 < \epsilon < T$ arbitrary.

\begin{figure}
\center
\def\svgwidth{\columnwidth}
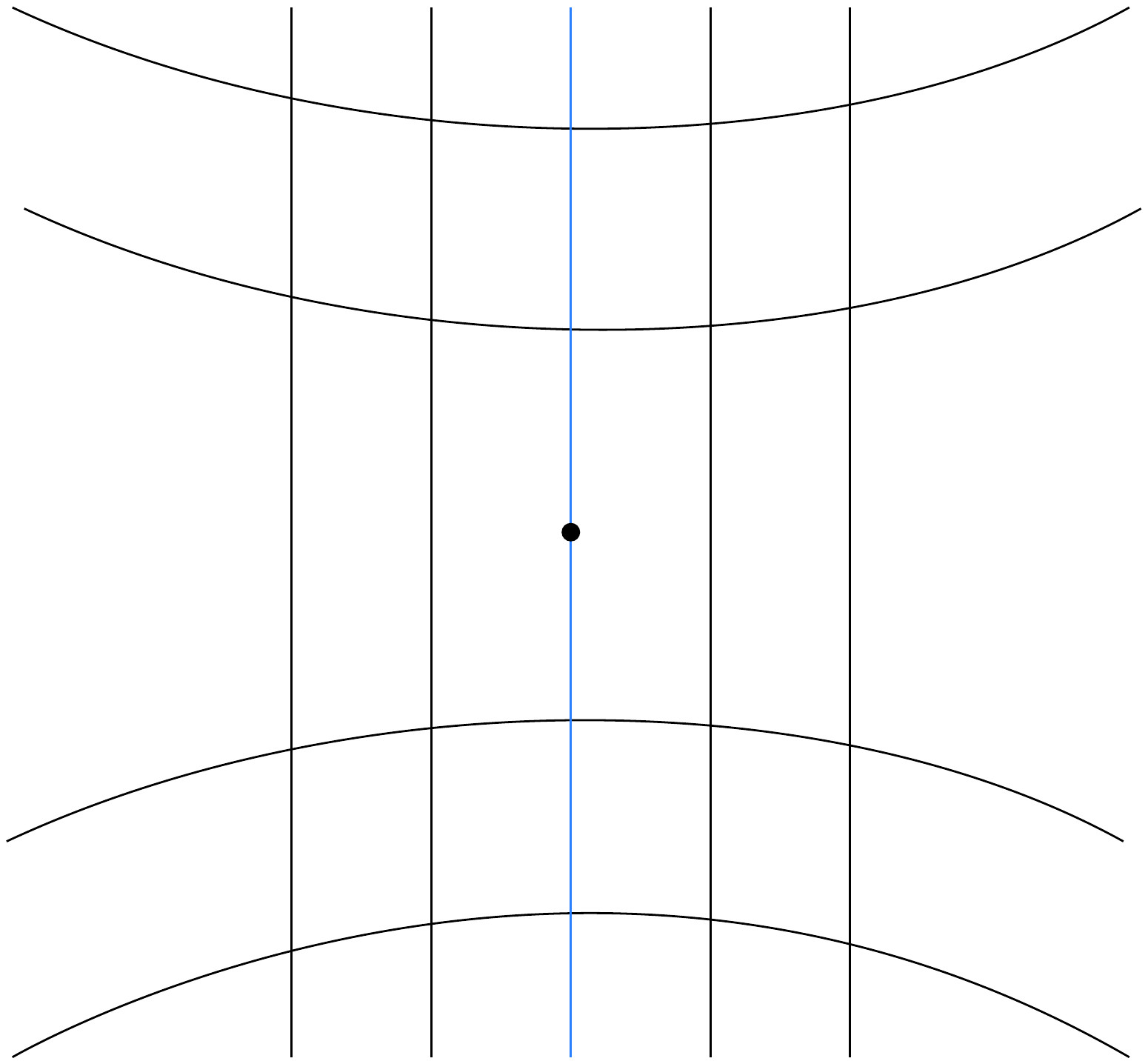
\caption{Values of $\eta_1, \eta_2$ on $\hS$}
\label{fig:s0supports}
\end{figure}

To choose $C, T$ appropriately, note that, by Lemma~\ref{lem:coordinates},
$$H_p \eta_1 = 2C\frac{\lambda}{\zeta} |\alpha|^2 + 2Cr + s,$$ where $r, s$ are homogeneous of order $m-1$ in $\zeta$, and $\iota^* s \in \IL^2, \iota^* r \in \IL^3$, where as before we let $\iota: \Sigma \cap U \hookrightarrow U$ be inclusion. Choose $C$ so that $C \frac{\lambda}{\zeta}|\alpha|^2 + s$ is of the opposite sign as $\lambda$ on of $\Sigma \cap U$). Then choose $T > 0$ sufficiently small so that $H_p \eta_1$ is of the opposite sign as $\lambda$ on $\supp(\tilde \chi(\eta_1) \tilde \chi(\eta_2) \tilde \chi(\eta_0^2))$, whose image under $\kappa$ is a compact subset of $U_0'$. We then let
\begin{align*}
\chi_0 & = \tilde \chi (\eta_0^2), \\
\chi_1 & = \tilde \chi( \eta_1), \\
\chi_2 & = \tilde \chi(\eta_2).
\end{align*}

We then define
\begin{align*}
	g_{1, t} & = \hat \rho(\zeta) \chi_2 \chi_3 \rho_t \sqrt{\sgn(\lambda) \chi_1 H_p \chi_1}  \\
	g_{2, t} & = \hat \rho(\zeta)\chi_1 \chi_2 \chi_3 \sqrt{\sgn(\lambda) (\rho_t H_p \rho_t + \sigma_{m-1}(\frac{P - P^*}{2i}) \rho_t^2)} \\
	e_t & = \hat \rho^2 \chi_1^2 \chi_3^2 \chi_2 H_p \chi_2 + \hat \rho^2\chi_1^2 \chi_2^2 \chi_3 H_p \chi_3 + \chi_1^2 \chi_2^2 \chi_3^2 \rho_t^2 \hat \rho H_p \hat \rho
\end{align*}
in $U$, and extend these to all of $X$ as identically $0$ outside of $U$. Note that the above choices of $C, T$, and $\epsilon$ ensure that $g_{1, t}$ and $g_{2, t}$ are smooth and real-valued, and that $q \in \mathrm{Ell}(g_{2, 0})$. The above choices also ensure the desired essential support for $e_t$, and we have
$$\frac{1}{2}H_p b_t^2 + a b_t^2 = \sgn(\lambda)(g_{1, t}^2 + g_{2, t}^2) + e_t.$$
Lastly, we can set $$h_t = \frac{b_t^2}{g_{2, t}} = \hat \rho \chi_1 \chi_2 \chi_3 \frac{\rho_t^2}{\sqrt{\sgn(\lambda)(\rho_t H_p \rho_t + \sigma_{m-1}(\frac{P - P^*}{2i}) \rho_t^2)}}$$ inside $U$, and identically $0$ outside of $U$. The symbols thus have the desired properties. \qed

\subsection{Proof of Lemma~\ref{lem:s1operators}} \label{sect:s1operatorconstruction}

It suffices to produce symbols
\begin{align*}
b = (b_t)_{t \in [0, 1]} & \in L^\infty([0, 1]_t, S^\frac{2s - m + 1}{2}(T^*X)), \\
g_1 = (g_{1, t})_{t \in [0, 1]}, g_2 = (g_{2, t})_{t \in [0, 1]} & \in L^\infty([0,1]_t, S^s(T^*X)), \\
e = (e_t)_{t \in [0, 1]} & \in L^\infty([0, 1]_t, S^{2s}(T^*X)), \\
h = (h_t)_{t \in [0, 1]} & \in L^\infty([0, 1]_t, S^{s - m + 1}(T^*X))
\end{align*} 
such that
\begin{align*}
\frac{1}{2} H_p b_t^2 + \sigma_{m-1}(\frac{P - P^*}{2i}) b_t^2 & = - \sgn (\lambda) (g_{1, t}^2 + g_{2, t}^2) + e_t \\
b_t^2 & = g_{2, t} h_t
\end{align*}
up to order $-\infty$, with:
\begin{enumerate}
	\item for $t > 0$, $b_t \in S^\frac{2s_1 - m + 1}{2}(T^*X), g_{j, t} \in S^{s_1}(T^*X), e_t \in S^{2s_1}(T^*X)$, and $h_t \in S^{s_1 - m +1}(T^*X)$,
	\item $b_t$ and $g_{j, t}$ are continuous in the topologies of $S^\frac{2s - m + 1 + \delta}{2}(T^*X)$ and \\ $S^{s + \delta}(T^*X)$, respectively, for all $\delta > 0$,
	\item all symbols are supported in $\kappa^{-1}(U_0)$,
	\item $\mathrm{esssup}(e_t) \cap \hS = \emptyset$,
	\item all symbols real-valued,
	\item $q \in \mathrm{Ell}(g_{2, 0})$. 
\end{enumerate}

We then quantize as in Section~\ref{sect:s0operatorconstruction}. To construct $b_t$, we again assume (by shrinking $U = \kappa^{-1}(U_0)$ if necessary) that $U$ has a coordinate chart $\phi$ as in Lemma~\ref{lem:coordinates}. We choose functions $\chi_0, \chi_1 \in C^\infty(V)$ homogeneous of degree $0$, and $$\rho_t \in L^\infty([0, 1]_t, S^\frac{2s - m + 1}{2}(X))$$ to serve similar roles as in Section~\ref{sect:s0operatorconstruction}. As in Section~\ref{sect:s0s1formulas}, we let $\rho_t = \zeta^\frac{2s - m + 1}{2} (1 + t\zeta)^{s_1 - s}$. As in our definition of $s_1$, choose an open neighborhood $U_0' \subset U_0$ of $q$, along with $\zeta_0 \in \RR_+$, so that
$$\rho_t H_p \rho_t + \sigma_{m-1}\bigl(\frac{P - P^*}{2i}\bigr) \rho_t^2$$
remains the opposite sign of $\lambda$ inside $\kappa^{-1}(U_0') \cap \zeta^{-1}((\zeta_0, \infty))$. We then take $\hat \rho: U \rightarrow \RR$ to be as in Section~\ref{sect:s0operatorconstruction}. We then let
$$b_t = \hat \rho \chi_0 \chi_1 \rho_t$$
inside $U$ and identically $0$ outside $U$. This will have the desired properties if:
\begin{itemize}
	\item $\sqrt{-\sgn(\lambda) \chi_1 H_p \chi_1}$ is real-valued and smooth.
	\item $\kappa(\supp(\chi_0 \chi_1))$ is a compact subset of $U_0'$.
	\item $\supp \chi_1 H_p \chi_0 \cap \Sigma = \emptyset$
\end{itemize}

To construct $\chi_0$ and $\chi_1$, let $\eta_1: U \rightarrow \RR$ be as before, but this time we will take $C > 0$. Let $\tilde \chi \in C^\infty(\RR)$ be as before, with $T$ to be chosen. We again have
$$H_p \eta_1 = 2 C \frac{\lambda}{\zeta} |\alpha|^2 + 2Cr + s$$ with $r, s$ homogeneous of order $m-1$ in $\zeta$, and $\iota^*  s \in \IL^2, \iota^* r \in \IL^3$ (as before $\iota: \Sigma \cap U \rightarrow U$ is inclusion). Choose $C > 0$ so that $C\frac{\lambda}{\zeta}|\alpha|^2 + s$ has the same sign as $\lambda$ on $\Sigma \cap U$. Then choose $T > 0$ sufficiently small so that $H_p \eta_1$ has the same sign as $\lambda$ on $\supp(\tilde \chi(\eta_1) \tilde \chi(\eta_0^2)$, whose image under $\kappa$ is a compact subset of $U_0'$. We then set, as in Section~\ref{sect:s0operatorconstruction}, $\chi_0 = \tilde \chi(\eta_0^2)$ and $\chi_1 = \tilde \chi(\eta_1)$.

We then let
\begin{align*}
	g_{1,t} & = \hat \rho \chi_0 \sqrt{-\sgn(\lambda) \chi_1 H_p \chi_1} \\
	g_{2,t} & = \hat \rho \chi_1 \chi_0\sqrt{-\sgn(\lambda)( \rho_t H_p  \rho_t  + \sigma_{m-1}(\frac{P - P^*}{2i}) \rho_t^2)} \\
	e_t & = \hat \rho^2 \chi_1^2 \rho_t^2 \chi_0 H_p \chi_0  
\end{align*}
in $U$, and extend these to all of $X$ as identically $0$ outside $U$. Note that the above choices of $C, T$, and $\epsilon$ ensure that ensure that $g_{1, t}$ and $g_{2, t}$ are real-valued and smooth, and that $q \in \mathrm{Ell}(g_{2, 0})$. The above choices also ensure the desired essential support for $e_t$, and we have $$\frac{1}{2} H_p b_t^2 + a b_t^2 = - \sgn(\lambda)(g_{1, t}^2 + g_{2, t}^2) + e_t$$ up to order $-\infty$. We leave out $\chi_0^2 \chi_1^2 \rho_t^2 \hat \rho H_p \hat \rho$ for convenience in adapting this to the proof of Theorem~\ref{thm:tau}.
 
Lastly, we can set $$h_t = \frac{b_t^2}{g_{2, t}} = \hat \rho \chi_0 \chi_1 \frac{\rho_t^2}{\sqrt{-\sgn(\lambda)(\rho_t H_p \rho_t + \sigma_{m-1}(\frac{P - P^*}{2i}))}}$$
on $U$ and identically $0$ outside of $U$. The symbols thus have the desired properties. \qed

\section{Proof of Theorem~\ref{thm:tau}} \label{sect:tauproof}

In the previous proofs, we constructed an operator $B_t$ such that $\frac{1}{2i}(B_t^2 P - P^* B_t^2)$ had some desired properties. The fact that we actually have a squared operator in that expression did not come into play much, and in fact was not needed. Here, however, the extra arrangement shall pay off.

\subsection{Sink Case} \label{sect:tausink}

Using the operator definitions as in section 4.1, let $B = B_0$, $G_j = G_{j,0}$, $E = E_0$, $M = M_0$, $F = F_0$, and $N = N_0$. Then for all $\tau \in [0, 1]$, we have
\begin{align*}
	\frac{1}{2i}(B^2(P - i Q_\tau) - (P^* + i Q_\tau)B^2) & = \frac{1}{2i}(B^2 P - P^* B^2) - \frac{1}{2}(B^2 Q_\tau + Q_\tau B^2) \\
	& = \sgn(\lambda)(G_1^*G_1 + G_2^* G_2) - B Q_\tau B \\
& \qquad + E + F  + \frac{1}{2}[[B, Q_\tau], B]\\
	& = - G_1^*G_1 - G_2^* G_2 - B Q_\tau B + E + F \\
& \qquad + \frac{1}{2}[[B, Q_\tau], B]
\end{align*}
where we used the fact that since $q$ is a sink, $\lambda < 0$. We can assume (by induction) that $\WF_{L^\infty([0, 1])}^{s - \frac{1}{2}}(u_\tau) \cap U_0 = \emptyset$. This time we further choose $U_0$ to be disjoint from $\WF^{s - m + 1}_{L^\infty([0, 1])}((P - i Q_\tau)u_\tau)$ and  $(\WF^s_{L^\infty([0, 1])}(u_\tau) \backslash L) \cap \hS$. The latter can be arranged by Corollary~\ref{cor:geometriccortau}. For $\tau > 0$, we pair with $u_\tau$ as before:
\begin{align*}
\frac{1}{2i}\lrangle{u_\tau, (B^2 (P - i Q_\tau) - (P^* + i Q_\tau) B^2)u_\tau} = & -\lrangle{u_\tau, G_1^*  G_1 u_\tau} - \lrangle{u_\tau, G_2^* G_2 u_\tau} \\
& - \lrangle{u_\tau, B Q_\tau B u_\tau} + \lrangle{u_\tau, E u_\tau}\\
& + \lrangle{u_\tau, (F  + \frac{1}{2}[[B, Q_\tau], B])u_\tau}.
\end{align*}

Note that for $\tau > 0$, these are all well-defined: since $P - i Q_\tau$ is elliptic for $\tau > 0$, by elliptic regularity, $V \cap \WF^{s + 1}(u_\tau) = \emptyset$. Hence $$\lrangle{u_\tau, G_j^2 u_\tau} = \|G_j u_\tau\|^2,$$ and  $$\lrangle{u_\tau, B Q_\tau B u_\tau} = \lrangle{B u_\tau, Q_\tau B u_\tau}$$ are well-defined. By the regularity assumption on $u_\tau$, $$\WF^s_{L^\infty([0, 1])}(u_\tau) \cap \WF'(E) = \emptyset,$$ so $\lrangle{u_\tau, E u_\tau}$ is well-defined and remains bounded as $\tau \rightarrow 0$. By our inductive assumption $\WF^{s - \frac{1}{2}}_{L^\infty([0, 1])}(u_\tau) \cap V = \emptyset$, along with the fact that $F, [[B, Q_\tau], B] \in \Psi^{2s - 1}(X)$, $$\lrangle{u_\tau, (F + \frac{1}{2}[[B, Q_\tau], B]) u_\tau}$$ is well-defined and remains bounded as $\tau \rightarrow 0$. Further, for $\tau > 0$, $$\frac{1}{2i}\lrangle{u_\tau, (B^2(P - i Q_\tau) - (P^* + i Q_\tau)) u_\tau } = \mathrm{Im}(\lrangle{u_\tau, B^2 (P - i Q_\tau) u_\tau})$$ is well-defined, and as before, we have $$|\mathrm{Im}\lrangle{u_\tau, B^2(P - i Q_\tau) u_\tau}| \leq |\mathrm{Im} \lrangle{u_\tau, N (P - i Q_\tau) u_\tau}| + \frac{c}{2}\|G_2 u_\tau\|^2 + \frac{1}{2c}\|M (P - i Q_\tau) u_\tau\|^2$$ for any $c > 0$. By the regularity assumptions on $u_\tau$ and $(P - i Q_\tau) u_\tau$, both $|\mathrm{Im}\lrangle{u_\tau, N(P - i Q_\tau) u_\tau}|$ and $\|M(P - i Q_\tau)u_\tau\|^2$ remain bounded as $\tau \rightarrow 0$.
\begin{align*}
\|G_1 u_\tau\|^2 + &(1 - \frac{c}{2})\|G_2 u_\tau\|^2 + \lrangle{B u_\tau, Q_\tau B u_\tau} \\ & \leq |\mathrm{Im}\lrangle{u_\tau, N(P - i Q_\tau)u_\tau}| + \frac{1}{2c}\|M(P - i Q_\tau) u_\tau\|^2 + |\lrangle{u_\tau, E u_\tau}| \\ & \quad + |\lrangle{u_\tau, (F + \frac{1}{2}[[B, Q_\tau], B])u_\tau}|
\end{align*}
Since $Q_\tau$ is positive semidefinite, $\lrangle{B u_\tau, Q_\tau B u_\tau} \geq 0$, so since all terms on the right hand side remain bounded, all terms on the left hand side remain bounded, and the proof proceeds as in earlier cases.\qed

\subsection{Source Case} \label{sect:tausource}

As we assume no a priori regularity on $u_\tau$ (as we assumed $q \notin \WF^{s_1}(u)$ in the previous theorem, for instance) the argument carries over to this theorem only after some extra preparation. Specifically, we can no longer work by induction, improving regularity by $\frac{1}{2}$ at each step. Since $u_\tau \in L^\infty([0, 1], \mathcal D'(X))$, we only have that $u_\tau \in L^\infty([0, 1], H^{-N}(X))$ for some $N$, but if we were to run the commutator argument and attempt to get regularity $-N + \frac{1}{2}$, the sign of the $G_2$ term would oppose the sign of $B Q_\tau B$, and so we cannot control the sum of these terms. Thus we will instead be more careful with our operator construction and ensure that $F, J \in \Psi^{-\infty}(X)$. As we are controlling errors which vary in $\tau$, we should expect that more of our operators depend on $\tau$. Below, $G_2$, $H$, $F$ and $J$ become $\tau$-dependent operators. While we are making things more precise, we might as well construct $G_1$ and $G_{2, \tau}$ to be self-adjoint along with $B$.

We will construct operators so that we have
\begin{align*}
\frac{1}{2i}(B^2 (P - i Q_\tau) - (P^* + i Q_\tau)B^2) & = - \sgn(\lambda)(G_1^2 + G_{2, \tau}^2) - B Q_\tau B + E + F_\tau \\
& = -G_1^2 - G_{2, \tau}^2 - B Q_\tau B + E + F_\tau \\
B^2 & = H_\tau G_{2, \tau} + J_\tau
\end{align*}
with 
\begin{itemize}
\item $B \in \Psi^\frac{2s - m + 1}{2}(X)$ with $B^* = B$,
\item $G_{1} \in \Psi^{s}(X)$ with $G_1^* = G_1$,
\item $G_2 = (G_{2, \tau})_{\tau \in [0, 1]} \in L^\infty([0,1]_\tau, \Psi^{s}(X))$ with $q \in \mathrm{Ell}_{L^\infty[0, 1]}(G_{2, \tau})$ with $G_\tau^* = G_\tau$,
\item $E \in \Psi^{2s}(X)$ with $\WF'(E) \cap \Sigma = \emptyset$, and 
\item $F_\tau \in L^\infty([0, 1]_\tau, \Psi^{-\infty}(X))$.
\item $J = (J_\tau)_{\tau \in [0, 1]} \in L^\infty([0, 1]_\tau, \Psi^{-\infty}(X))$, and
\item $H = (H_\tau)_{\tau \in [0, 1]} \in L^\infty([0, 1]_\tau, \Psi^{s - m + 1}(X))$ 
\end{itemize}
We want similar supports as before: all operators have $\WF'$ (in the case of operators varying in $\tau$, $\WF'_{L^\infty([0, 1])}$) contained in $U_0$, chosen so that $\WF_{L^\infty([0, 1])}(P - iQ_\tau) \cap U_0 = \emptyset$. 

Assuming this, the argument proceeds as usual: we obtain
\begin{align*}
	\|G_1 u_\tau\|^2 + & (1 - \frac{c}{2})\|G_{2, \tau} u_\tau\|^2 + \lrangle{B u_\tau, Q_\tau B u_\tau} \\ & \leq |\mathrm{Im} \lrangle{u_\tau, J_\tau(P - i Q_\tau) u_\tau}| + \frac{1}{2c}\|H_\tau(P - i Q_\tau) u_\tau\|^2 \\ & \quad + |\lrangle{u_\tau, E u_\tau}| + |\lrangle{u_\tau, F_\tau u_\tau}|
\end{align*}
Since $J_\tau, F_\tau \in L^\infty([0,1]_\tau,\Psi^{-\infty}(X))$, $\lrangle{u_\tau, J(P - i Q_\tau) u_\tau}$ and $\lrangle{u_\tau, F u_\tau}$ remain bounded as $\tau \rightarrow 0$. By the regularity assumption on $(P - i Q_\tau) u_\tau$, $\|H(P - i Q_\tau) u_\tau\|$ remains bounded as well. The proof proceeds as in the previous proofs, and we obtain $G_\tau u_\tau \in L^\infty([0, 1], L^2(X))$. Thus $q \notin WF^s_{L^\infty([0, 1])}(u_\tau)$.

Now we must construct these operators. As we need extra control, we must be more careful in specifying our quantization map $q$. Essentially, since we are working in a coordinate neighborhood, it suffices to use the standard Weyl quantization (we could use another quantization, but we want some operators to be self-adjoint, and this makes that easier) in that neighborhood, and the corresponding full symbol map. To be more precise, let $\pi: T^*X \rightarrow X$ be projection to the base. By shrinking $U$, we may assume that there is an open $U_X' \subset X$ of $\overline{\pi(U)}$ and canonical coordinate chart $\psi: U' = \pi^{-1}(U_X') \rightarrow V' \subset \RR^n_x \times \RR^n_\xi$. Let $\psi_X: U_X' \rightarrow \pi(V')$ be the corresponding map for the base. Let $g \in C^\infty_c(X, \RR)$ be identically $1$ in $\pi(U)$ and supported inside $\pi(U')$. We define a quantization 
$$q: S^r_0(U) \rightarrow \Psi^r(X)$$ as follows, where $S^r_0(U)$ is the space of symbols on $X$ whose support is contained in $U$. Given $a \in S^r_0(U)$ and $v \in C^\infty(X)$, define 
$$q(a)v = g \psi_X^*( q_W((\psi^{-1})^* a)(\psi_X^{-1})^*(gv)),$$ extended as identically $0$ outside of $U'$ (implicitly in the formula, we extend $(\psi_X^{-1})^*(gv)$ as $0$ outside $U'_X$, and we extend $(\psi^{-1})^* a$ as $0$ outside of $U'$). Further, we have a full symbol map $$\sigma: \Psi^r(X) \rightarrow S^r(\RR_x^n; \RR_\xi^n),$$ defined as follows. Given $A \in \Psi^r(X)$, we may associate with it with an element of $\Psi^r(\RR_x^n)$ by $v \mapsto (\psi_X{^-1})^*(gA(g\psi_X^*v))$, extending as identically $0$ outside of $U_X'$. We then use the standard Weyl full symbol map on this operator. 

This quantization and corresponding symbol map have the following properties. First, given $A \in \Psi^r(X)$, $\sigma_r(A)|_U$ has as a representative $\sigma(A)|_U$. Second, $\psi^* \circ \sigma \circ q$ is the identity on $S_0^r(U)$, at least after extending the image of this map to be identically $0$ outside of $U$. Third, if we choose a density on $X$ which agrees with the standard density on $\RR^n_x$ when pulled back by $\psi_X^{-1}$, then if $a \in S^r_0(U)$ is real-valued, $q(a)$ is self-adjoint. Fourth, $$\WF'(A) \cap U \subseteq \psi^{-1}(\mathrm{esssup}(\sigma(A)))$$ for any $A \in \Psi^r(X)$. Fifth, given $A \in \Psi^r(X)$ and $B \in \Psi^{r'}(X)$, then we have the following asymptotic expansion, valid only inside $\phi(U)$:
$$\sigma(A \circ B)(x, \xi) \sim \sum^\infty_{j = 0} \frac{1}{j!} \{\sigma(A), \sigma(B)\}_j(x, \xi)$$
where $\{a, b\}_j(x, \xi) := (\frac{i}{2})^j (D_\xi \cdot D_y - D_x \cdot D_\eta)^j a(x, \xi) b(y, \eta)|_{y = x, \eta = \xi}$

In what follows, we leave out pullbacks by $\psi$ and $\psi^{-1}$ so as to avoid cluttered formulas. Let $b = b_0, e = e_0, g_1 = g_{1, 0}, g_{2, s} = g_{2, 0}$ as defined in Section~\ref{sect:s1operatorconstruction} (so all supported within $U$), with an additional condition on $\tilde{\chi}:\RR \rightarrow \RR$. In some small neighborhood of $T$, we would like $\tilde{\chi}(t) = \exp(-\frac{1}{T - t})$ for $t < T$, and $ \tilde{\chi}(t) = 0$ for $t \geq T$. The details do not matter so much; it simply achieves what we really need:
\begin{itemize}
	\item $\tilde{\chi}'(t) = r(t)\tilde{\chi}(t)$ on $t < T$ for some rational function $r$ which is smooth for $t < T$.
	\item $s(t)\tilde \chi(t)$ is smooth for any rational function $s$ which is smooth on $t < T$.
\end{itemize}

We then let $B = q(b), E = q(e), G_1 = q(g_1)$, and the strategy will be to include lower-order terms for $G_2$ to cancel error terms.  We proceed to choose real-valued $g_{2, s - j} \in L^\infty([0, 1], S_0^{s-j}(U))$ ($g_{2, s}$ has already been chosen and is $\tau$-independent - hence if $g_{2, s}$ is elliptic at $q$, $q \in \Ell_{L^\infty([0, 1])}(G_{2, \tau})$) and then create real-valued $g_2 \in L^\infty([0,1], S_0^s(U))$ with asymptotic expansion 
\begin{equation} \label{gerrorreduce1}
g_2 \sim \sum_{j = 0}^\infty g_{2, s - j}
\end{equation}
so that if $G_{2, \tau} = q(g_2)$, then $F_\tau \in L^\infty([0,1]_\tau, \Psi^{-\infty}(X))$. Note that if each $g_{2, s - j}$ is real-valued, then $g_2$ can be chosen to be real valued. Thus $B, E, G_1,$ and $G_{2,\tau}$ are self-adjoint. 

Let
\begin{align}
A & := \frac{1}{2i} (B^2(P - i Q_\tau) - (P^* + i Q_\tau) B^2) + G_1^2 + B Q_\tau B - E \label{gerrorreduce2}\\
& = \frac{1}{2i}(B^2 P - P^* B^2) + \frac{1}{2}[[B, Q_\tau], B] + G_1^2 - E \notag
\end{align}
We would thus like to choose $g_{2, s - j}$ so that $A + G_{2, \tau}^2 \in L^\infty([0, 1]_\tau, \Psi^{-\infty}(X))$. We have the following asymptotic expansion:
\begin{equation} \label{gerrorreduce3}
\sigma(A) \sim \sum^\infty_{j = 0} a_{2s - j}, a_{2s - j} \in S^{2s - j}(\RR_x^n \times \RR_\xi^n)
\end{equation}
where
\begin{align}
a_{2s - j} = &  \sum_{k + l = j + 1} \frac{1}{k! l!} \biggl( \frac{(\{\{b, b\}_l, \sigma(P)\}_k - \{\sigma(P^*), \{b, b\}_l\}_k)}{2i} \notag  \\
& \qquad + \frac{\{\{b, \sigma(Q_\tau)\}_l - \{\sigma(Q_\tau), b\}_l, b\}_k}{2} \notag \\
& \qquad - \frac{\{b, \{b, \sigma(Q_\tau)\}_l - \{\sigma(Q_\tau), b\}_l\}_k}{2}\biggr) \label{gerrorreduce5}\\
& + \sum_j \frac{\{g_1 , g_1\}_j}{j!} \notag \\
& + \delta_{0j} e \notag
\end{align}
up to order $-\infty$ - we leave out all terms where a derivative is applied to $\hat \rho$. 
Note that our earlier construction ensured that $a_{2s} = - g_{2, s}^2$, up to order $-\infty$. The specifics of this are not so important, except that each $a_{2s - 1 - j}$ is a sum of functions of the form $r \chi_0^2 \chi_1^2 g \hat \rho^2$, where $g$ is smooth and $r$ is a rational function with poles at $C|\beta|^2 + |\alpha|^2 = T$ and $\eta_0^2 = T$ (i.e., the boundary of $\supp( \chi_0 \chi_1)$). Denote this property by (*).

We define $g_{2, s - j}$ recursively for $j > 0$:
\begin{align}
	g_{2, s - j} = & - \frac{a_{2s - j}}{2g_{2, s}} - \sum_{0 < k + l \leq j, k > 0, l > 0} \frac{ \{g_{2, s - k}, g_{2, s - l}\}_{j - k - l} }{2g_{2, s} (j - k  - l)!} \label{gerrorreduce6}\\
	& - \sum_{0 < k < j} \frac{\{g_{2, s}, g_{2, s - k}\}_{j - k} + \{g_{2, s - k}, g_{2, s}\}_{j - k}}{2g_{2, s} (j - k)!} \notag
\end{align}
(up to order $-\infty$ - we again leave out all terms where a derivative is applied to $\hat rho$) where $g_{2, s} \neq 0$ and identically $0$ when $g_{2, s} = 0$. Note that this recursive definition makes sense: the definition for $g_{2, s - j}$ depends only on $g_{2, s - l}$ for $l < j$. Further, these are smooth: since $a_{2s - j}$ has property (*), and $g_{2, s}$ is $\chi_0 \chi_1 \hat \rho$ times a nonvanishing function, we may recursively check the numerator in the definition of $g_{2, s - j}$ always has property (*), using the properties of $\tilde \chi$.

Lastly, note that $A + G_2^2 \in L^\infty([0,1]_\tau,\Psi^{-\infty}(X))$: we have asymptotic expansion
\begin{align}
	\sigma(G_2^2 ) & = \sum_{k + l \leq j} \frac{ \{g_{2, s - k}, g_{2, s - l}\}_{j - k - l}}{(j - k - l)!} \notag \\
	& = 2g_{2, s}g_{2, s - j} + \sum_{k + l \leq j, k > 0, l > 0}  \frac{\{g_{2, s - k}, g_{2, s - l}\}_{j - k - l}}{(j - k - l)!} \label{gerrorreduce7} \\
	& \quad + \sum_{0 < k < j} \frac{\{g_{2, s}, g_{2, s - k}\}_{j - k} + \{g_{2, s - k}, g_{2, s}\}_{j - k}}{(j - k)!}, \notag
\end{align}
and each $g_{2, s - j}$ is chosen so that $2g_{2, s}g_{2, s - j}$ cancels out all other terms of order $2s - j$ in the asymptotic expansion for $A + G_2^2$.

To ensure that $J_\tau \in L^\infty([0, 1], \Psi^{-\infty}(X))$, we will construct $H_\tau$ in a similar way. That is, we will let $H_\tau = q_L(h_\tau)$, where 
\begin{equation} \label{herrorreduce1}
h_\tau \sim \sum^\infty_{j = 0} h_{s - m + 1 - j}, h_{s - m + 1 - j} \in L^\infty([0, 1], S^{s - m + 1 - j}_0(U)),
\end{equation}
defined recursively. For any such choice of $h$, we have
\begin{align}
	\sigma(B^2 - G_2 H) \sim & \sum^\infty_{j = 0} \biggl( - h_{s - m + 1 - j} g_{2, s} + \frac{\{b, b\}_j}{j!} \notag \\
	& \qquad - \sum_{k + l \leq j, l > 0}  \frac{\{h_{s - m + 1 - k}, g_{2, s - l}\}_{j - k - l}}{(j - k - l)!} \label{herrorreduce2}\\
	& \qquad - \sum_{k < j} \frac{\{h_{s - m + 1 - k}, g_{2, s}\}_{j - k}}{(j - k)!}\biggr) \notag
\end{align}

This gives us the formula for the recursive definition of $h_{s - m + 1 - j}$:
\begin{align}
	h_{s - m + 1 - j} = & \frac{\{b, b\}_j}{j! g_{2, s}} -  \sum_{k + l \leq j, l > 0}  \frac{\{h_{s - m + 1 - k}, g_{2, s - l}\}_{j - k - l}}{(j - k - l)! g_{2, s}} \label{herrorreduce3} \\
	& \qquad \quad - \sum_{k < j} \frac{\{h_{s - m + 1 - k}, g_{2, s}\}_{j - k}}{(j - k)!g_{2, s}} \notag
\end{align}
(up to order $-\infty$ - we again leave out all terms where a derivative is applied to $\hat \rho$) when $g_{2, s} \neq 0$, and $0$ otherwise.
As before, we may inductively conclude that the numerators in the formula for $h_{s - m + 1 - j}$ all have property (*), so this definition makes sense. Further, $h_{s - m + 1 - j}$ is defined so that $h_{s - m + 1 - j} g_{2, s}$ cancels out all other terms of order $2s - m + 1 - j$ in the asymptotic expansion for $B^2 - G_2 H$. This completes the proof. \qed

\section{Iterative Regularity} \label{sect:iterativeregularity}

Here we state and prove analogs/generalizations of the above in the context of Lagrangian regularity. This largely applies the discussion of \cite[Section 6]{hmv1}, as corrected in \cite[Appendix A]{hmv2}. We provide full details here instead of simply quoting the results, in part to translate from the scattering setting, and in party to slightly modify the assumptions.

We begin by defining this sense of regularity. Given $O \subset S^*X$, let 
$$\Psi^r(O) = \{A \in \Psi^r(X) \ | \ \WF'(A) \subset O\}.$$

\begin{define}(\cite[Definition 6.1]{hmv1})
A {\em test module} in an open set $O \subset S^*X$ is a linear subspace $\MM \subset \Psi^1(O)$ which (contains and) is a module over $\Psi^0(O)$, which is closed under commutators and which is finitely generated in the sense that there exist finitely many $A_i \in \Psi^1(X), 0 \leq i \leq N, A_0 = \Id$, such that each $A \in \MM$ can be written as
$$A = \sum^N_{i = 0} Q_i A_i, Q \in \Psi^0(O).$$
\end{define}

\begin{rmk}
The generators $A_i$ need not be in $\MM$. As $\Id$ is a generator, $\MM^0 = \Psi^0(O) \subset \MM \subseteq \MM^2 \ldots$.
\end{rmk}

\begin{define}(\cite[Definition 6.2]{hmv1})
Let $\MM$ be a test module in an open set $O \subset S^*X$. For $u \in C^{-\infty}(X)$ we say that $u \in I^{(s)}(O, \MM)$ if $\MM^k u \subset H^s(X)$ for all $k$. We say that $u \in I^{(s), k}(O, M)$ if $\MM^k u \subset H^s$.
\end{define}

Recall that $u \in \mathcal D'(X)$ is a Lagrangian distribution associated to Lagrangian submanifold $\Lambda$ if there exists $s$ such that for any $k$ and any $A_1, \ldots, A_k \in \Psi^1(X)$ with $\sigma_1(A_j)|_\Lambda = 0$, $$A_1 \ldots A_k u \in H^s.$$

We microlocalize this. Given $O \subset S^*X$, $P \in \Psi^m(X)$ with homogeneous principal symbol $p$, and a conic Lagrangian submanifold $\Lambda \subset \Sigma(P)$ such that $H_p$ is radial and nonvanishing on $\Lambda$, we let
$$\M(O) = \{A \in \Psi^1(O) \ | \ \sigma_1(A) \ | \ \Lambda = 0\}.$$

We verify that this is, in fact, a test module. That $\M$ is closed under commutators follows from the fact that $\Lambda$ is coisotropic, as if $a$ and $b$ are symbols which vanish on a given coisotropic submanifold, then $\{a, b\}$ also vanishes on this coisotropic submanifold. For the finite generation, we can assume that $\overline{O} \subset U_0$, with $U_0 = \kappa(U)$ as in Lemma~\ref{lem:coordinates}, as we can microlocalize around such neighborhoods, then patch together with a partition of unity. Let $\chi \in C^\infty(S^*X)$ be identically $1$ in $O$ and $0$ outside of $\kappa(U)$, and let $\hat \rho: U \rightarrow \RR$ be the cutoff as in the proof of Lemma~\ref{lem:s0operators} (so $\hat \rho$ vanishes in a neighborhood of the $0$-section of $T^*X$, and is identically $1$ for sufficiently large $\zeta$). We then let $A_i = q(\chi \hat \rho \alpha_i \zeta)$, $0 < i < n$, $A_0 = \Id$, and $A_n =  q(\chi \hat \rho \zeta^{1 - m}) P$, then $\M$ is generated by $A_i, 0 \leq i \leq n$. This is a principal symbol statement that follows from the fact that $\eta_0, \alpha_i, i = 1 \ldots n-1$ are defining functions for $\Lambda \cap U$.

We then have the following result.  As above, we take $U$ as in Lemma~\ref{lem:coordinates}.
\begin{thm} \label{thm:iterative}
Given $P \in \Psi^m(X)$ with a real-valued homogeneous principal symbol $p$ such that $H_p$ is radial (and nonvanishing) on a conic Lagrangian submanifold $\Lambda \subset \Sigma(P)$, then given $q \in \kappa(\Lambda)$ and $s_0, s_1$ as in Theorem~\ref{thm:general},
\begin{itemize}
	\item For $s < s_0$, if there is an open neighborhood $O'$ of $q$ such that $Pu \in I^{(s - m + 1), k}(O', \M(O'))$ and $\Gamma_q \cap \WF^{s + k}(u) \cap O' = \emptyset$, then there exists an open neighborhood $O \subset O'$ of $q$ such that $u \in I^{(s), k}(O, \M(O))$.
	\item For $s > s_1$, if there is an open neighborhood $O'$ of $q$ such that $Pu \in I^{(s - m + 1), k}(O', \M(O'))$ and $u \in I^{(s_1), k}(O', \M(O'))$, then there exists an open neighborhood $O \subset O'$ of $q$ such that $u \in I^{(s), k}(O, \M(O))$.
	\item For $s \geq s_1 + 1$, if there is an open neighborhood $O'$ of $q$ such that $Pu \in I^{(s - m + 1), k}(O', \M(O'))$ and $\WF^{s_1}(u) \cap O' = \emptyset$, there exists an open neighborhood $O \subset O'$ of $q$ such that $u \in I^{(s), k}(O, \M(O))$.
\end{itemize}
\end{thm}

\begin{rmk}
We expect that the following strengthening of part of Theorem~\ref{thm:iterative} is true.

{\it Given $P$, $\Lambda$, $s_1$, and $p$ as in Theorem~\ref{thm:iterative}, for $s > s_1$, if there is an open neighborhood $O'$ of $q$ such that $Pu \in I^{(s - m + 1), k}(O', \M(O'))$ and $\WF^{s_1}(u) \cap O' = \emptyset$, then there exists an open neighborhood $O \subset O'$ of $q$ such that $u \in I^{(s), k}(O, \M(O))$.}

The methods used here appear not to be able to deal with this statement - the difficulty comes in the second regularization below (\eqref{secondregularizerstart} - \eqref{secondregularizerend}) and in making sense of terms such as $\|G_{2, t} A u\|$ for even $t > 0$. This is perhaps a defect in our methods; as the definitions are, we can only make sense of $\M^k$ for $k$ a nonnegative integer. If we could inductively improve the orders of regularity with smaller intervals, the arguments might go more smoothly. It may be, however, that with a more clever regularization, this machinery could handle such a statement.

Indeed, this is particularly easy to see in the Fourier transform side of Melrose's setting \cite{melroseasymp} when the operator is classicall and $\lambda_0$ is constant along the Lagrangian, as one should be able to construct explicit solutions $v$ to $Pv = f$, with $v$ having the desired Lagrangian regularity. If this were done, we could compare $v$ to $u$ using Theorem~\ref{thm:general}, and obtain this stronger result.
\end{rmk}

\begin{proof}[Proof of Theorem~\ref{thm:iterative}]

In what follows, we can assume that $\overline{O'} \subset U_0$ with $U_0 = \kappa(U)$ as in Lemma~\ref{lem:coordinates}. In order to prove the above statement, it suffices to show that $G_\gamma A_\gamma \in L^2(X)$, where $\gamma \in \mathbb Z_{\geq 0}^{n-1}, |\gamma| \leq k$, $G_\gamma$ is elliptic on $O$, and 
$$A_\gamma = \prod^{n-1}_{i = 1} A^{\gamma_i}_i.$$
Note that we do not need to include products involving $A_n$, as they are already covered by the assumption on $Pu$. The order $A_1, \ldots, A_{n-1}$ is irrelevant, as products with different orders commute modulo lower order powers of the test module (see \cite[Lemma 6.3]{hmv1} for further details).

To prove the theorem, we use a positive commutator argument.  As we show below, positivity follows from the following property our module enjoys (see \cite[Eq. 6.15]{hmv1} for a more general condition under which such a statement might hold):
\begin{equation} \label{commutatorcondition}
\frac{1}{2i}[A_i, P] = \sum^n_{j = 0} C_{ij} A_j,
\end{equation}
for $i = 1 \ldots n-1$, with $C_{ij} \in \Psi^{m-1}(X)$ for $j = 0 \ldots n$, $\sigma_{m-1}(C_{ij})|_\Lambda = 0$ for $0 < j < n$. This is again a principal symbol statement which follows from Lemma~\ref{lem:coordinates}. From this it follows that, for $\gamma \in \mathbb Z_{\geq 0}^{n-1}$,
$$\frac{1}{2i} [A_\gamma, P] = R_\gamma + \sum_{\delta \in \mathbb Z_{\geq 0}^{n-1}, \ |\delta| = |\gamma|} C_{\gamma \delta} A_\delta,$$
where $C_{\gamma \delta} \in \Psi^{m-1}(X)$ with $\sigma_{m-1}(C_{\gamma \delta})|_{\Lambda} = 0$, and
\begin{equation} \label{iterativeerrorterm}
R_\gamma = \sum_{|\delta| < |\gamma|} D_{\gamma \delta} A_\delta + E_{\gamma \delta} A_\delta P,
\end{equation}
where $D_{\gamma \delta} \in \Psi^{m-1}(X)$ and $E_{\gamma \delta} \in \Psi^0(X)$.

We start with the $s < s_0$ case. We work by induction on $k$ - the base case is Theorem~\ref{thm:general}. Assuming $u \in I^{(s), k-1}(O'', \M(O''))$ for some neighborhood $O'' \subseteq O'$ of $q$, we take $B_t = q(b_t)$ as in the proof of Lemma~\ref{lem:s0operators}, with $\WF'(B) \subset O''$. Below, we will shrink the microsupport, but for now, let us simply look at what operator relations we have, given sufficiently small microsupport. We have
\begin{align*}
\frac{1}{2i} & \sum_{\gamma \in \mathbb Z_{\geq 0}^{n-1}, |\gamma| = k}  A_\gamma^* B_t^2 A_\gamma P - P^* A_\gamma^*  B_t^2 A_\gamma  \\
& = \sum_{\gamma \in \mathbb Z_{\geq 0}^{n-1}, |\gamma| = k} A_\gamma^*\frac{([B_t^2, P] + (P - P^*)B_t^2)}{2i} A_\gamma + A_\gamma^*B^2 \frac{[A_\gamma, P]}{2i}  + \frac{[A_\gamma^*, P^*]}{2i}B_t^2 A_\gamma  \\ 
& = \sum_{\gamma \in \mathbb Z_{\geq 0}^{n-1}, |\gamma| = k} A_\gamma^*\frac{([B_t^2, P] + (P - P^*)B_t^2)}{2i} A_\gamma + R_\gamma^* B_t^2 A_\gamma + A_\gamma^* B_t^2 R_\gamma \\
& \quad + \sum_{\delta \in \mathbb Z_{\geq 0}^{n-1}, |\delta| = k} A_\gamma^* B_t^2 C_{\gamma \delta} A_\delta  + A^*_\delta C_{\gamma \delta}^* B_t^2 A_\gamma \\
& = A^*\biggl( \frac{[B_t^2, P] + (P - P^*)B_t^2}{2i} + B_t^2 C + C^* B_t^2 \biggr)A + R^*B_t^2 A + A^* B_t^2 R
\end{align*}
where in the last line, we let $A = (A_\gamma)$ and $R = (R_\gamma)$ be column vectors, running over all $\gamma \in \mathbb Z_{\geq 0}$ with $|\gamma| = k$, and $C = (C_{\gamma \delta})$ a matrix of operators (or rather an operator on sections of a trivial bundle over $X$).

Using the symbols as in the proof of Lemma~\ref{lem:s0operators}, we have, up to order $2s - 1$,
\begin{align*}
\sigma_{2s}(\frac{[B_t^2, P] + (P - P^*)B_t^2}{2i} + B_t^2 C +  C^* B_t^2)_{\gamma \delta} & = \sgn(\lambda)(g_{1, t}^2 + g_{2, t}^2)\delta_{\gamma \delta} + e_t \delta_{\gamma \delta} \\
& \quad + b_t^2 \sigma_{m-1}(C_{\gamma \delta} + C_{\delta \gamma}^*)
\end{align*}
(the notation might be a little confusing - $\delta_{\gamma \delta}$ is the Kronecker delta with indices $\gamma$ and $\delta$) where $g_{2, 0}$ is elliptic of order $s$ in a neighborhood of $q$. As $\sigma_{m-1}(C_{\gamma \delta} + C_{\delta \gamma}^*)|_\Lambda = 0$, $\sgn(\lambda)g_{2, t}^2 + b_t^2 \sigma_{m-1}(C_{\gamma \delta} + C_{\delta \gamma}^*)$ is elliptic and $\sgn(\lambda)$-definite in a neighborhood of $q$. Thus, if we choose the support of $b$ to be sufficiently small, we have
\begin{align}
A^*\biggl( & \frac{[B_t^2, P]   + (P - P^*)B_t^2}{2i} + B_t^2 C + C^* B_t^2 \biggr)A \label{g2absorb}\\
& = A^*(\sgn(\lambda)(G_{1, t}^*G_{1, t} + G_{2,t}^* G_{2,t})  + E_t + F_t)A \notag
\end{align}
where $G_2$, $E_t$, and $F_t$ are matrices of operators, with $G_{2, 0}$ elliptic in a neighborhood of $q$, $\WF'_{L^\infty[0, 1]}(E_t) \cap \kappa(\Lambda) = \emptyset$, and $F_t$ uniformly (in $t$) of order $-\infty$. This last point is done to avoid using the two-step induction needed in the correction \cite[Appendix A]{hmv2}. The same argument as in \eqref{gerrorreduce1}-\eqref{gerrorreduce7} works here, as the $A$ factors, $\tau$ dependence in the previous setting, $t$ dependence here, and that these are matrices of operators, are irrelevant for the construction. Further, we can choose matrices of operators $H_t$ and $J_t$, uniformly (in $t$) of orders $s - m + 1$ and $-\infty$, respectively, so that
$$B_t^2 = H_t G_{2, t} + J_t$$
as on the level of principal symbols, $B_t$ and $G_{2, t}$ have the same cosphere cutoff functions. That $J_t$ can be made uniformly of order $-\infty$ uses the same argument as \eqref{herrorreduce1}-\eqref{herrorreduce3}.

We proceed with the positive commutator argument in the standard way. For $t > 0$,
\begin{align*}
\frac{1}{2i} (\lrangle{B_t^2 Au, A P u} - \lrangle{A Pu, B_t^2 Au}) & = \sgn(\lambda)(\|G_{1, t} A u\|^2 + \|G_{2, t} A u\|^2) \\
& \quad + \lrangle{Au, E_t Au} + \lrangle{Au, F_t Au} \\
& \quad + \lrangle{R u, B_t^2 A u} + \lrangle{B_t^2 A u, R u}.
\end{align*}

As mentioned above, we assume that $u \in I^{(s), k-1}(O'', \M(O''))$, and we take $\WF'_{L^\infty[0, 1]}(B_t) \subset O''$ and $\WF'(A_i) \subset O''$ for each $i$. We can then bound $\|G_{j, t} A u\|$ as $t \rightarrow 0$, using essentially the same considerations as in previous commutator arguments, with slight modifications. We bound the $\lrangle{R u, B_t^2 A u}$ and $\lrangle{B_t^2 A u, R u}$ terms with a familiar method. For $t > 0$, 
\begin{align*}
|\lrangle{Ru, B_t^2 A u}| & = |\lrangle{H_t R u, G_{2, t} A u} + \lrangle{J_t R u, A u}|\\
& \leq 2\|H_t R u\|^2 + \frac{1}{2} \|G_{2, t} A u\|^2 +  |\lrangle{J_t R u, A u}|
\end{align*}
The $G_{2, t} A u$ term can be absorbed into the other such term we are trying to bound, the $H_t R u$ term can be bounded using the inductive hypothesis and the form \eqref{iterativeerrorterm} of each entry of the vector of operators $R$, and the last term has $J_t$ uniformly of order $-\infty$.

We handle the $s_0 + 1 > s > s_0$ and $s \geq s_0 + 1$ cases together. We again inductively assume that $u \in I^{(s), k - 1}(O'', \M(O''))$, and have
\begin{align*}
	\frac{1}{2i}(A^* B_t^2 A P - P^* A^* B_t^2 A) & = A^*(-\sgn(\lambda)(G_{1, t}^* G_{1, t} + G_{2, t}^* G_{2, t}) + E_t + F_t)A \\
& \quad + R^* B_t^2 A + A^* B_t^2 R
\end{align*}
where now we let $B_t = q(b_t)$ with $b_t$ as in the proof of Lemma~\ref{lem:s1operators}. $G_{2, t}, E_t,$ and $F_t$ are matrices of operators, with $G_{2, 0}$ elliptic of order $s$ and $G_{2, t}$ of order $s_1$ for $t > 0$. $\WF'_{L^\infty[0, 1]}(E_t) \cap \Sigma(P) = \emptyset$, and $F_t$ is uniformly of order $-\infty$ (we again need to use the technique of the proof of the source case of Theorem~\ref{thm:tau}, but again this carries over with little change, so we provide no further details here). As before, we also arrange that
$$B_t^2 = H_t G_{2, t} + J_t$$
with $H_t$ and $J_t$ uniformly of orders $s - m + 1$ and $-\infty$, respectively. We take all operators constructed (including each $A_i$) to have microsupport contained in $O''$.

To proceed with the positive commutator estimate, we introduce a second regularizer as in the proof of Lemma~\ref{lem:integratebyparts} (and for similar reasons - as it stands, we have not yet made sense of terms such as $\lrangle{u, P^* A^* B_t^2 A u}$). Let $(A_\tau) \in L^\infty([0, 1]_\tau, \Psi^0(X)$ be such that $A_\tau \in \Psi^{-\infty(X)}$ for $\tau > 0$, and $A_\tau \rightarrow \Id$ as $\tau \rightarrow 0$ in $\Psi^\delta(X)$ for $\delta > 0$. For $t, \tau > 0$, we can then make sense of
\begin{align}
	\frac{1}{2i}(\lrangle{A_\tau u & , A^* B_t^2 A P u}  - \lrangle{A_\tau u, P^* A^* B_t^2 A u}) = \notag \\
& -\sgn(\lambda)(\lrangle{A_\tau u, A^* G_{1, t}^* G_{1, t} A} + \lrangle{A_\tau u, A^* G_{2, t}^* G_{2, t} A u}) \label{secondregularizerstart}\\
& \lrangle{A_\tau u, (A^* E_t A + A^* F_t A + R^* B_t^2 A + A^* B_t^2 R) u}. \notag
\end{align}
We can then manipulate each term and send $\tau \rightarrow 0$. For instance, for $t, \tau > 0$,
\begin{align}
	\lrangle{A_\tau u, P^* A^* B_t^2 A u} & = \lrangle{A [P, A_\tau] u, B_t^2 A u} + \lrangle{A A_\tau P u, B_t^2 A u}.
\end{align}
$A [P, A_\tau]$ is uniformly a vector of operators in $\Psi^{m-1}(X) \M(O'')^k$.  We then have $A [P, A_\tau] u$ a vector of distributions in $H^{s - m}(X)$, uniformly in $\tau$. For the $s_1 + 1 > s > s_1$ case, by assumption we have $A u$ a vector of distributions in $H^{s_1}(X)$, and for $s \geq s_1 + 1$, we have $Au$ a vector of distributions in $H^{s-1}(X)$, as $A$ is a vector of operators in $\Psi^1(X) \M(O'')$. In either situation, we can, as in the proof of Lemma~\ref{lem:integratebyparts}, take $\tau \rightarrow 0$, and in the limit we get $\lrangle{APu, B_t^2 A u}$. All other terms follow similarly, and we have
\begin{align}
\frac{1}{2i} (\lrangle{B_t^2 Au, A P u} - \lrangle{A Pu, B_t^2 Au}) & = - \sgn(\lambda)(\|G_{1, t} A u\|^2 + \|G_{2, t} A u\|^2)  \notag \\
& \quad + \lrangle{Au, E_t Au} + \lrangle{Au, F_t Au} \label{secondregularizerend}\\
& \quad + \lrangle{R u, B_t^2 A u} + \lrangle{B_t^2 A u, R u}. \notag
\end{align}
We then take $t \rightarrow 0$ as before, completing the proof.

\end{proof}

We also have an analogue of Theorem~\ref{thm:tau} in this iterative regularity setting. In order to have a transparent statement, we impose a technical condition: that $Q_\tau$ has homogeneous principal symbol $q_\tau$ of the form
\begin{equation} \label{extraqtaucondition}
q_\tau = \tau \nu. 
\end{equation}
By assumption, $\nu$ is homogeneous and elliptic around the point of interest $q$.

\begin{thm} \label{thm:tauiterative}
Given $P, Q, \Lambda, q, s_0,$ and $s_1$ as in the statement of Theorem~\ref{thm:tau} along with the additional assumption \eqref{extraqtaucondition}, and $u = (u_\tau) \in L^\infty([0, 1]_\tau, \mathcal D'(X))$,
\begin{itemize}
	\item if $\kappa(\Lambda)$ is a sumbanifold of sinks for $W_p|_{S^*X}$, then for $s < s_0$, the existence of an open neighborhood $O'$ such that
$$(P - i Q) u \in L^\infty([0, 1]_\tau, I^{(s - m + 1), k}(O', \M(O'))$$ and
$$\Gamma_q \cap \WF^{s + k}(u) \cap O' = \emptyset$$
implies that for some open neighborhood $O \subset O'$ of $q$, $$u \in L^\infty([0, 1]_\tau, I^{(s), k}(O, \M(O)).$$
	\item if $\kappa(\Lambda)$ is a submanifold of sources for $W_p|_{S^*X}$, then for $s > s_1$, the existence of an open neighborhood $O'$ such that
$$(P - i Q) u \in L^\infty([0, 1]_\tau, I^{(s - m + 1), k}(O', \M(O'))$$
implies that for some open neighborhood $O \subset O'$ of $q$, $$u \in L^\infty([0, 1]_\tau, I^{(s), k}(O, \M(O)).$$
\end{itemize}
\end{thm}

To prove this, we adapt the $A_i$ to work well with $Q_\tau$. We would like, for $0 < i < n$,
\begin{equation} \label{qtauarrangement}
	[A_i, Q] \in L^\infty([0, 1]_\tau, \Psi^{m-1}) \M
\end{equation}
for the positive commutator argument below. We do this by altering the coordinates chosen in Lemma~\ref{lem:coordinates}, and then using the above definition of $A_i$ with the new coordinates. We begin as before, choosing canonical coordinates $x, \xi$ such that locally, $\Lambda$ is $N^*\{x_n = 0\}$, with $\xi_n > 0$, and $\kappa^{-1}(q) \{x = 0, \xi_i = 0, i < n\}$. For convenience, we let $x_n = z, y = (y_1, \ldots ,y_n) = (x_1 \ldots x_{n-1})$, and $\xi' = (\xi_1 \ldots \xi_{n-1})$. By \eqref{extraqtaucondition} and the positive-semidefiniteness of $Q_\tau$, we have 
$$\nu = \xi_n^m \gamma(z, y, \frac{\xi'}{\xi_n})^m$$
 locally, with $\gamma > 0$.
Note that if we set
\begin{align*}
y & = \tilde{y} \\
z & = \tilde{z} \gamma(0, y, 0) \\
\xi_n & = \frac{\tilde{\xi}_n}{\gamma(0, y, 0)} \\
\xi_i & = \tilde{\xi}_i - \tilde{\xi}_n z \frac{\partial_{y_i} \gamma(0, y, 0)}{\gamma^2}
\end{align*}
then $\tilde{z}, \tilde y, \tilde \xi$ are canonical coordinates, and locally, $\Lambda = N^*\{\tilde z = 0\}$ and $\kappa^{-1}(q) = \{\tilde z = 0, \tilde y = 0, \tilde \xi = 0, i < n\}$. Further,
$$\nu = \xi_n (1 + f)$$
with $f \in \IL$, where $\IL$ is as defined immediately before the statement of Lemma~\ref{lem:coordinates}, and $U$ is an open set on wich these coordinates are defined. We can thus proceed with further choices of coordinates as in Lemma~\ref{lem:coordinates}, and we get \eqref{qtauarrangement}. Further, if we write, for $0 < i < n$,
$$[A_i, Q_\tau]  = \sum^n_{j = 0} C_{ij, \tau} A_j,$$
we have that
\begin{equation} \label{qtauccondition}
	\sigma_{m-1}(C_{ij, \tau}) \stackrel{\tau \rightarrow 0}{\rightarrow} 0 \ \mathrm{in} \ S^m(X)/S^{m-1}(X).
\end{equation}

The proof then follows as a modification of the above. We use a positive commutator estimate using ``commutator''
\begin{equation} \label{taucommutator}
\frac{1}{2i}(A^* B^2 A (P - i Q_\tau) - (P^* + iQ_\tau) A^* B^2 A).
\end{equation}
\eqref{qtauccondition} allows, for sufficiently small $\tau$, all terms involving $Q_\tau$, other than the positive-semidefinite $A^*B Q_\tau B A$, to be absorbed into the $G_2$ matrix of operators as in \eqref{g2absorb}, and so we have that \eqref{taucommutator} is equal to
\begin{align*}
 - A^*(G_1^* G_1 + G_2^* G_2)A  - A^* B Q_\tau B A + A^* E A  + A^*F A + R^* B^2 A + A^* B^2 R
\end{align*}
where $F$ is uniformly of order $-\infty$. As in the proof of Theorem~\ref{thm:tau}, $Q_\tau$ regularizes for us, so $B$ need have no regularization. $\WF'(E)$ is, in the sink case, where we assume regularity, and in the source case, off the characteristic set. The proof proceeds by induction as above.

\bibliographystyle{plain}

\bibliography{lagrangian_radial_version2}

\end{document}